\documentclass[11pt]{article}
\usepackage{amsfonts}
\usepackage{latexsym}
\usepackage{amsmath}
\usepackage{amssymb}
\usepackage{amsthm}
\usepackage{amsmath}
\usepackage{hyperref}
\usepackage{verbatim}
\usepackage{bibspacing}

\newtheorem{thm}{Theorem}[section]
\newtheorem*{thm*}{Theorem}
\newtheorem{cor}[thm]{Corollary}

\newtheorem{lemma}[thm]{Lemma}
\newtheorem{prop}[thm]{Proposition}
\newtheorem*{prop*}{Proposition}
\newtheorem{proposition}[thm]{Proposition}

\newtheorem*{conj*}{Conjecture}

\newtheorem*{defn*}{Definition}
\theoremstyle{definition}
\newtheorem{rem}[thm]{\textbf{Remark}}
\newtheorem*{rmk*}{Remark}
\newtheorem*{fact*}{Fact}

\theoremstyle{proof}

\newcommand{\Var}{\textrm{Var}}

\newcommand{\norm}[1]{\left\Vert#1\right\Vert}

\newcommand{\abs}[1]{\left\vert#1\right\vert}
\newcommand{\set}[1]{\left\{#1\right\}}
\newcommand{\brac}[1]{\left(#1\right)}
\newcommand{\scalar}[1]{\left \langle #1 \right \rangle}
\newcommand{\sscalar}[1]{\langle #1 \rangle}
\newcommand{\Real}{\mathbb{R}}
\newcommand{\R}{\mathbb{R}}
\newcommand{\N}{\mathbb{N}}

\newcommand{\E}{\mathbb{E}}
\newcommand{\EE}{\mathcal{E}}

\newcommand{\Id}{\text{Id}}

\renewcommand{\H}{\mathcal{H}}
\newcommand{\eps}{\epsilon}
\newcommand{\K}{\mathcal{K}}
\renewcommand{\S}{\mathcal{S}}

\newcommand{\lin}{\text{lin}}
\newcommand{\lambdaC}{\lambda^C}
\newcommand{\Leb}{\mathfrak{m}}
\newcommand{\extr}{\text{extr}}
\newcommand{\supp}{\text{\rm supp}}
\newcommand{\n}{\mathfrak{n}}

\newlength{\defbaselineskip}
\numberwithin{equation}{section}

\begin{document}

\title{A sharp centro-affine isospectral inequality of Szeg\"{o}--Weinberger type and the $L^p$-Minkowski problem}
\date{}

\author{Emanuel Milman\textsuperscript{1}}

\footnotetext[1]{Department of Mathematics, Technion - Israel
Institute of Technology, Haifa 32000, Israel. 
The research leading to these results is part of a project that has received funding from the European Research Council (ERC) under the European Union's Horizon 2020 research and innovation programme (grant agreement No 637851). Email: emilman@tx.technion.ac.il.}

\begingroup    \renewcommand{\thefootnote}{}    
\endgroup

\maketitle

\begin{abstract}
We establish a sharp upper-bound for the first non-zero eigenvalue corresponding to an \emph{even} eigenfunction of the Hilbert--Brunn--Minkowski operator 
(the centro-affine Laplacian) associated to a strongly convex $C^2$-smooth \emph{origin-symmetric} convex body $K$ in $\R^n$. 
Our isospectral inequality is centro-affine invariant, attaining equality if and only if $K$ is a (centered) ellipsoid; this is reminiscent of the (non affine invariant) classical Szeg\"{o}--Weinberger isospectral inequality for the Neumann Laplacian. The new upper-bound complements the conjectural lower-bound, which has been shown to be equivalent to the log-Brunn--Minkowski inequality and is intimately related to the uniqueness question in the even log-Minkowski problem. 
As applications, we obtain new strong non-uniqueness results for the \emph{even} $L^p$-Minkowski problem in the subcritical range $-n < p < 0$, as well as new rigidity results for the critical exponent $p=-n$ and supercritical regime $p < -n$. In particular, we show that \emph{any} $K$ as above, which is not an ellipsoid, is a witness to non-uniqueness in the even $L^p$-Minkowski problem for all $p \in (-n,p_K)$ and some $p_K \in (-n,0)$, and that $K$ can be chosen so that $p_K$ is arbitrarily close to $0$. 
\end{abstract}

\section{Introduction}

A central question in contemporary Brunn--Minkowski theory is that of existence and uniqueness in the $L^p$-Minkowski problem for $p \in (-\infty,1)$: given a finite non-negative Borel measure $\mu$ on the unit-sphere $S^{n-1}$, determine conditions on $\mu$ which ensure the existence and/or  uniqueness of a convex body $K$ in $\R^n$ so that:
\begin{equation} \label{eq:intro-Lp-Minkowski}
h_K^{1-p} S_K  = \mu  . 
\end{equation}
Here $h_K$ and $S_K$ denote the support function and surface-area measure of $K$, respectively -- we refer to Section \ref{sec:prelim} for standard missing definitions. When $h_K \in C^2(S^{n-1})$, 
\[
S_K = \det(D^2 h_K) \Leb,
\]
where $\Leb$ is the induced Lebesgue measure on $S^{n-1}$, $D^2 h_K = \nabla^2_{S^{n-1}} h_K + h_K \delta_{S^{n-1}}$ and $\nabla_{S^{n-1}}$ is the Levi-Civita connection on $S^{n-1}$ with its standard Riemannian metric $\delta_{S^{n-1}}$. Consequently, (\ref{eq:intro-Lp-Minkowski}) is a Monge--Amp\`ere-type equation. 
It describes self-similar solutions to the (anisotropic) $\alpha$-power-of-Gauss-curvature flow for $\alpha = \frac{1}{1-p}$ \cite{Andrews-GaussCurvatureFlowForCurves,Andrews-FateOfWornStones,Andrews-PowerOfGaussCurvatureFlow,AndrewsGuanNi-PowerOfGaussCurvatureFlow,BCD-PowerOfGaussCurvatureFlow,ChoiDaskalopoulos-UniquenessInLpMinkowski,Chow-PowerOfGaussCurvatureFlow,Urbas-PositivePowersOfGaussCurvatureFlow,Urbas-NegativePowersOfGaussCurvatureFlow}.

The case $p=1$ above corresponds to the classical Minkowski problem of finding a convex body with prescribed surface-area measure; when $\mu$ is not concentrated on any hemisphere and its barycenter is at the origin, existence and uniqueness (up to translation) of $K$ were established by 
Minkowski, Alexandrov and Fenchel--Jessen (see \cite{Schneider-Book-2ndEd}), and regularity of $K$ was studied by Lewy \cite{Lewy-RegularityInMinkowskiProblem}, Nirenberg \cite{Nirenberg-WeylAndMinkowskiProblems}, Cheng--Yau \cite{ChengYau-RegularityInMinkowskiProblem}, Pogorelov \cite{Pogorelov-MinkowskiProblemBook}, Caffarelli \cite{CaffarelliHigherHolderRegularity,Caffarelli-StrictConvexity} and many others. 
The extension to general $p$ was put forth and publicized by E.~Lutwak \cite{Lutwak-Firey-Sums} as an $L^p$-analog of the Minkowski problem for the $L^p$ surface-area measure $S_p K = h_K^{1-p} S_K$ which he introduced. 
Existence and uniqueness in the class of origin-symmetric convex bodies (``the \emph{even} $L^p$-Minkowski problem"), when the measure $\mu$ is even and not concentrated in a hemisphere, was established for $n \neq p > 1$ by Lutwak \cite{Lutwak-Firey-Sums} and for $p = n$ by Lutwak--Yang--Zhang \cite{LYZ-LpMinkowskiProblem}. A key tool in the range $p \geq 1$ is the prolific $L^p$-Brunn--Minkowski theory, initiated by Lutwak \cite{Lutwak-Firey-Sums,Lutwak-Firey-Sums-II} following Firey \cite{Firey-Sums},
and developed by Lutwak--Yang--Zhang (e.g. \cite{LYZ-LpAffineIsoperimetricInqs,LYZ-SharpAffineLpSobolevInqs,LYZ-LpJohnEllipsoids}) and others, which extends the classical $p=1$ case. Further existence, uniqueness and regularity results in the range $p > 1$ under various assumptions on $\mu$ were obtained in \cite{ChouWang-LpMinkowski,GuanLin-Unpublished,HuangLu-RegularityInLpMinkowski,HLYZ-DiscreteLpMinkowski,LutwakOliker-RegularityInLpMinkowski,Zhu-ContinuityInLpMinkowski}.

The case $p < 1$ turns out to be more challenging because of the lack of an appropriate $L^p$-Brunn--Minkowski theory. Existence, (non-)uniqueness and regularity under various conditions on $\mu$ were studied by numerous authors when $p<1$ (from either side of the critical exponent $p=-n$), especially after the important work by Chou--Wang \cite{ChouWang-LpMinkowski}, see e.g. \cite{BBC-SmoothnessOfLpMinkowski,BBCY-SubcriticalLpMinkowski,BHZ-DiscreteLogMinkowski,BoroczkyHenk-ConeVolumeMeasure,ChenEtAl-LocalToGlobalForLogBM,ChenLiZhu-LpMongeAmpere,ChenLiZhu-logMinkowski,HeLiWang-MultipleSupercriticalLpMinkowski, JianLuZhu-UnconditionalCriticalLpMinkowski,LuWang-CriticalAndSupercriticalLpMinkowski,StancuDiscreteLogBMInPlane,Stancu-UniquenessInDiscretePlanarL0Minkowski,Stancu-NecessaryCondInDiscretePlanarL0Minkowski,Zhu-logMinkowskiForPolytopes,Zhu-CentroAffineMinkowskiForPolytopes,Zhu-LpMinkowskiForPolytopes,Zhu-LpMinkowskiForPolytopesAndNegativeP}. 
The case $p=0$ is of particular importance as it corresponds to the \emph{log-Minkowski problem} for the cone-volume measure
\[
V_K := \frac{1}{n} h_K S_K ,
\]
obtained as the push-forward of the cone-measure on $\partial K$ onto $S^{n-1}$ via the Gauss map; note that the total mass of $V_K$ is $V(K)$, the volume of $K$. 
Being a self-similar solution to the isotropic Gauss-curvature flow, the case $p=0$ and $\mu = \Leb$ of (\ref{eq:intro-Lp-Minkowski}) describes the ultimate fate of a worn stone in a model proposed by Firey \cite{Firey-ShapesOfWornStones} and further studied in \cite{Andrews-FateOfWornStones,AndrewsGuanNi-PowerOfGaussCurvatureFlow,BCD-PowerOfGaussCurvatureFlow,ChoiDaskalopoulos-UniquenessInLpMinkowski,Kolesnikov-OTOnSphere}.

In \cite{BLYZ-logMinkowskiProblem}, B\"or\"oczky--Lutwak--Yang--Zhang showed that an \emph{even} measure $\mu$ is the cone-volume measure $V_K$ of an \emph{origin-symmetric} convex body $K$ if and only if it satisfies a certain subspace concentration condition, thereby completely resolving the existence part of the \emph{even} log-Minkowski problem. Uniqueness in the even log-Minkowski problem is known (when $K$ is not a parallelogram) for $n=2$ \cite{GageLogBMInPlane,StancuDiscreteLogBMInPlane,BLYZ-logBMInPlane,MaLogBMInPlane}, but remains open for $n \geq 3$. As put forth by B\"or\"oczky--Lutwak--Yang--Zhang in their influential work \cite{BLYZ-logMinkowskiProblem,BLYZ-logBMInPlane} and further developed in \cite{KolesnikovEMilman-LocalLpBM}, the uniqueness question is intimately related to the validity of a conjectured $L^0$- (or log-)Brunn--Minkowski inequality for origin-symmetric convex bodies, which constitutes a strengthening of the classical $p=1$ case. The restriction to origin-symmetric bodies is natural, and necessitated by the fact that no $L^p$-Brunn--Minkowski inequality can hold for general convex bodies when $p < 1$. For additional information and partial results on the conjectured log-Brunn--Minkowski inequality, we refer to \cite{BoroczkyDe-StableLogBMWithSymmetries,BoroczkyKalantz-LogBMWithSymmetries,BLYZ-logBMInPlane,ChenEtAl-LocalToGlobalForLogBM,ColesantiLivshyts-LocalpBMUniquenessForBall,CLM-LogBMForBall,HKL-LogBMForSubsets,Kolesnikov-OTOnSphere,KolesnikovLivshyts-ParticularFunctionsInLogBM,KolesnikovEMilman-LocalLpBM,LMNZ-BMforMeasures,MaLogBMInPlane,Putterman-LocalToGlobalForLpBM,Rotem-logBM,Saroglou-logBM1,Saroglou-logBM2,XiLeng-DarAndLogBMInPlane}. 

\medskip

In our previous joint work with A.~Kolesnikov \cite{KolesnikovEMilman-LocalLpBM}, we embarked on a systematic study of the validity of the \emph{local} $L^p$-Brunn--Minkowski inequality for origin-symmetric convex bodies and $p < 1$; by ``local" we mean on an infinitesimal scale, or equivalently, for pairs of bodies which are close enough to each other in an appropriate sense. To that end, we introduced the following elliptic second-order differential operator on $C^2(S^{n-1})$, called the Hilbert--Brunn--Minkowski operator $\Delta_K$, defined for $K \in \K^2_+$, the collection of convex bodies in $\R^n$ having the origin in their interior with $C^2$-smooth boundary and strictly positive curvature. In a local frame on $S^{n-1}$, $\Delta_K$ is given by:
\begin{align*}
\Delta_K z & := ((D^2 h_K)^{-1})^{ij} D^2_{ij}( z h_K) - (n-1) z \\
& =  ((D^2 h_K)^{-1})^{ij} ( h_K (\nabla^2_{S^{n-1}} z)_{ij}  + \partial_i h_K \partial_j z + \partial_j h_K \partial_i z)  
\end{align*}
 (our original definition, denoted by $L_K$, differed by a factor of $n-1$ from the present one). Up to different normalization and gauge transformations, 
  $\Delta_K$ coincides with the operator introduced by Hilbert in his proof of the Brunn--Minkowski inequality (see \cite{BonnesenFenchelBook}). Very recently, we have shown \cite{EMilman-IsomorphicLogMinkowski} that $\Delta_K$ coincides with the centro-affine Laplacian of $\partial K$, parametrized on $S^{n-1}$ via the Gauss map. 
 
The operator $-\Delta_K$  is symmetric and positive semi-definite on $L^2(V_K)$, admitting a unique self-adjoint extension with compact resolvent. Its spectrum thus consists of a countable sequence of eigenvalues of finite multiplicity starting at $0$ and tending to $\infty$. It was shown in \cite{KolesnikovEMilman-LocalLpBM} that $\Delta_{K}$ enjoys a remarkable centro-affine equivariance property, stating that for any $T \in GL_n$, $\Delta_{T(K)}$ and $\Delta_K$ are conjugates modulo an isometry of Hilbert spaces; in particular, the spectrum $\sigma(-\Delta_{T(K)})$ is the same for all $T$. 
It is immediate to check that $\Delta_K+(n-p) \Id $ is precisely the linearization of $\log( h_K^{1-p} \det(D^2 h_K))$ appearing in the left-hand-side of (\ref{eq:intro-Lp-Minkowski}) under a logarithmic variation $h_{K_\eps} = h_K(1 + \eps \; \cdot)$. Consequently, understanding whether $n-p$ is in the spectrum of $-\Delta_K$ is of fundamental importance to the uniqueness question in the $L^p$-Minkowski problem. For similar reasons, Hilbert realized that the classical Brunn--Minkowski inequality (the case $p=1$) \cite{Schneider-Book-2ndEd} is equivalent to the statement that $\sigma(-\Delta_{K}) \cap (0,n-1) = \emptyset$, and proved that indeed 
$\lambda_1(-\Delta_K) = n-1$ where $\lambda_1$ denotes the first non-zero eigenvalue \cite{BonnesenFenchelBook}.

Let $\K$ denote the collection of convex bodies in $\R^n$ containing the origin in their interior, and let $\K_e$ denote those elements which are origin-symmetric. Similarly, let $\K^2_{+,e}$ and $C^2_e(S^{n-1})$ denote the origin-symmetric / even elements of $\K^2_+$ and $C^2(S^{n-1})$, respectively. 
Given $K \in \K^{2}_{+,e}$, we defined in \cite{KolesnikovEMilman-LocalLpBM} the first non-zero \emph{even} eigenvalue of $-\Delta_K$ (corresponding to an \emph{even} eigenfunction) as:
\[
\lambda_{1,e}(-\Delta_K)  := \inf \set{ \frac{\int_{S^{n-1}} (-\Delta_K z) z dV_K}{\int_{S^{n-1}} z^2 dV_K - \frac{(\int_{S^{n-1}} z dV_K)^2}{V(K)}} \;  ; \;  \text{non-constant } z \in C^2_{e}(S^{n-1})  } .
\]
It was shown in \cite{KolesnikovEMilman-LocalLpBM} that for any $p < 1$, the statement $\lambda_{1,e}(-\Delta_K) \geq n-p$ is equivalent to the \emph{local} $L^p$-Brunn--Minkowski inequality for origin-symmetric perturbations of $K$, and implies the \emph{local} uniqueness for the even $L^q$-Minkowski problem for any $q > p$. The fact that a \emph{local} verification of these problems is enough to imply the \emph{global} one was subsequently shown by Chen--Huang--Li--Liu for the uniqueness of the $L^p$-Minkowski problem \cite{ChenEtAl-LocalToGlobalForLogBM} and by Putterman for the $L^p$-Brunn--Minkowski inequality \cite{Putterman-LocalToGlobalForLpBM}. The conjecture is that $\lambda_{1,e}(-\Delta_K) > n$ for all $K \in \K^2_{+,e}$, which would confirm the log-Brunn--Minkowski inequality in $\K_e$ and the uniqueness in the $L^p$-Minkowski problem in $\K^2_{+,e}$ for all $p \in [0,1)$. 
As explained in \cite{KolesnikovEMilman-LocalLpBM}, the $n$-dimensional cube $Q^n = [-1,1]^n$ is the extremal case in this conjecture (at least formally, since $Q^n \notin \K^{2}_{+,e}$); specifically, it was shown that $\underline{\lambda}_{1,e}(Q^n) = n$ (see the definition of $\underline{\lambda}_{1,e}$ below).

Our main result in \cite{KolesnikovEMilman-LocalLpBM} was showing that $\lambda_{1,e}(-\Delta_K) \geq n-p_0$ for $p_0=1 - \frac{c}{n^{3/2}}$ and all $K \in \K^2_{+,e}$, which together with the aforementioned local-to-global results yields the verification of the above major problems for $p \in (p_0,1)$. In fact, thanks to recent progress on the KLS conjecture due to Chen \cite{Chen-AlmostKLS} and Klartag--Lehec \cite{KlartagLehec-AlmostKLS}, our estimate from \cite[Corollary 6.8 and Theorem 6.9]{KolesnikovEMilman-LocalLpBM} immediately improves to $p_0 = 1 - \frac{c}{n^{1+o(1)}}$. 

\subsection{Sharp isospectral upper-bound of Szeg\"{o}--Weinberger type }

In this work, instead of focusing on bounding $\lambda_{1,e}(-\Delta_K)$ from below, we tackle the opposite question of bounding it from above. We confirm the following sharp isospectral property of $-\Delta_K$, conjectured in \cite[Conjecture 5.15]{KolesnikovEMilman-LocalLpBM}, together with the corresponding equality cases:
\begin{thm} \label{thm:main1}
For all $K \in \K^2_{+,e}$, we have:
\[
\lambda_{1,e}(-\Delta_K) \leq 2 n ,
\]
with equality if and only if $K$ is a (centered) ellipsoid. 
\end{thm}

Indeed, when $K = B_2^n$, the unit Euclidean ball in $\R^n$, $\Delta_{B_2^n}$ coincides with the Laplace-Beltrami operator on $S^{n-1}$ with its standard Riemannian metric, and so the first non-zero even eigenvalue of $-\Delta_{B_2^n}$ is precisely $2n$, corresponding to the eigenspace of (even) quadratic harmonic polynomials in $\R^n$ (restricted to $S^{n-1}$). 
Since the spectrum of $-\Delta_K$ is centro-affine invariant, the same applies to all (centered) ellipsoids. Note that a similar characterization of ellipsoids using the first non-zero eigenvalue $\lambda_1(-\Delta_K)$ would \emph{fail}, since (as shown by Hilbert and already mentioned above) the latter is always equal to $n-1$ for \emph{any} $K \in \K^2_+$. 

While countless geometric characterizations of ellipsoids are known in the literature (see e.g. \cite{Soltan-EllipsoidsSurvey} for a survey), we are not aware of any prior \emph{spectral} characterization as above. Compare with \cite[Section 4]{Simon-LaplacianInAffineGeometry} for a characterization of ellipsoids as equality cases in certain inequalities involving several geometric parameters including the first non-zero Laplace-Beltrami eigenvalue on $(\partial K,g_{B})$ where $g_B$ is the Blaschke equiaffine metric, and \cite[Section 8]{LSSW-HigherOrderCodazzi} for a characterization of ellipsoids in terms of a PDE involving the Laplace-Beltrami operator on $(\partial K,\text{III})$ where $\text{III}$ denotes the Euclidean third fundamental form (which is isometric to the canonical unit-sphere $(S^{n-1},\delta_{S^{n-1}})$ via the Gauss map).

\bigskip

Although we do not see any direct relation, it might still be insightful to compare Theorem \ref{thm:main1} to the classical Szeg\"{o}--Weinberger isospectral inequality. Let $\Delta^N_\Omega$ denote the Neumann Laplacian on a sufficiently smooth bounded domain $\Omega \subset \R^n$. It was shown by Szeg\"{o} \cite{Szego-IsospectralNeumannIn2D} (for simply-connected domains in $\R^2$) and Weinberger \cite{Weinberger-IsospectralNeumann} (for general domains in $\R^n$), that the first non-zero Neumann eigenvalue $\lambda_1(-\Delta^N_\Omega)$ satisfies:
\[
\lambda_1(-\Delta^N_\Omega) \leq \lambda_1(-\Delta^N_{\Omega^*}) , 
\]
where $\Omega^*$ denotes a Euclidean ball having the same volume as $\Omega$; equality occurs if and only if $\Omega$ is a Euclidean ball. 
Note that this isospectral problem is not affine-invariant, and so the only maximizer is a Euclidean ball (as opposed to all ellipsoids), and moreover, one has to fix its volume as well. 

The reverse question regarding minimizing $\lambda_1(-\Delta^N_\Omega)$ while constraining the volume of $\Omega$ does not make sense, as it is easy to see that it can be made arbitrarily small by choosing $\Omega$ with necks, or alternatively, by stretching out $\Omega$. The first problem can be remedied by only considering convex domains $K$. The second one is again a manifestation of the lack of affine-invariance of this problem. However, if one considers the affine-invariant parameter $\lambda^N_{1,\text{aff}}(K) := \sup_{T \in GL_n} V(T(K))^{2/n} \lambda_1(-\Delta^N_{T(K)})$, the minimization question does make sense. A celebrated conjecture of Kannan--Lov\'asz--Simonovits  (KLS) \cite{KLS} (in combination with its known relation to the Slicing Problem \cite{BallNguyen-KLSImpliesSlicing,EldanKlartagThinShellImpliesSlicing}) predicts that for all convex bodies $K$ in $\R^n$:
\[
\lambda^N_{1,\text{aff}}(K) \geq c ,
\]
for some universal constant $c > 0$ independent of the dimension $n$. We refer to \cite{Chen-AlmostKLS,KlartagLehec-AlmostKLS} and the references therein for the best known estimate on $c = c_n$ and the history of this conjecture. As already mentioned, the analogous question in our setting is whether $\lambda_{1,e}(-\Delta_K) \geq n$ for all $K \in \K^2_{+,e}$, which is equivalent to the conjectured log-Brunn--Minkowski inequality in $\K_e$. 

One point of similarity between Theorem \ref{thm:main1} and the Szeg\"{o}--Weinberger theorem is that in both cases one only needs to find a good test function to upper bound the spectral parameter (by the Rayleigh--Ritz characterization). The test function used by Weinberger is one of the first $n$ non-trivial Neumann eigenfunctions of $-\Delta^N_{\Omega^*}$, appropriately centered and fitted onto $\Omega$. A crucial point in Weinberger's argument is finding an appropriate center by employing a fixed-point argument. In our setting, everything is origin-symmetric and so there is no need for centering. On the other hand, we do not know how to simply use the first non-trivial even eigenfunctions of $-\Delta_{B_2^n}$, given by quadratic harmonic polynomials, as test functions for $\lambda_{1,e}(-\Delta_K)$.

\subsection{$K$-adapted linear functionals}

However, the above analogy is useful if one extends the notion of ``quadratic polynomial" and adapts it to the given convex body $K$. Hilbert's original definition of his differential operator had the $n$-dimensional space of linear functions $\scalar{\cdot,\xi}$ ($\xi \in \R^n$) as the first non-trivial eigenfunctions. With our definition of $-\Delta_K$, the first eigenfunctions corresponding to the first non-zero eigenvalue $n-1$ are in fact:
\[
\lin_{K,\xi} := \frac{\scalar{\cdot,\xi}}{h_K} ; 
\]
 we will say that they are ``$K$-adapted linear functions". Instead of using quadratic polynomials per-se as even test functions in Theorem \ref{thm:main1}, in the form $\scalar{\cdot,\xi}^2$ or $\scalar{\cdot,\xi}^2/h_K$, we will use $\lin_{K,\xi}^2 = \scalar{\cdot,\xi}^2 / h_K^2$, the square of our $K$-adapted linear functions. Our proof of Theorem \ref{thm:main1} relies on the following new ingredient, a type of strengthened Cauchy-Schwarz inequality, which is of independent interest:
 
\begin{thm} \label{thm:main-direction}
For any convex body $K \in \K$ there exists $\xi \in S^{n-1}$ so that:
\begin{equation} \label{eq:main-direction}
\int_{S^{n-1}} \lin_{K,\xi}^4 dV_K \geq \frac{3n}{n+2} \frac{(\int_{S^{n-1}} \lin_{K,\xi}^2 dV_K)^2}{V(K)} . 
\end{equation}
If the support of $S_K$ is the entire $S^{n-1}$, equality in (\ref{eq:main-direction}) for all $\xi \in S^{n-1}$ holds if and only if $K$ is a (centered) ellipsoid. 
\end{thm}

For the proof of Theorem \ref{thm:main-direction}, we consider an appropriate linear image of $K$ (``position") and randomly select the direction $\xi \in S^{n-1}$ according to the uniform Haar measure. Interestingly, the relevant position is the one for which the $L^2$-surface area measure $S_2 K$ is isotropic -- a position introduced by Lutwak--Yang--Zhang in \cite{LYZ-NewEllipsoid} under the name ``dual isotropic" and further studied in \cite{LYZ-LpJohnEllipsoids, ZouXiong-IdenticalJohnAndLYZEllipsoids,HuXiong-LogJohnEllipsoid}.

\subsection{Implications for the even $L^p$-Minkowski problem} \label{subsec:intro-LpMinkowski}

Our results have several implications regarding non-uniqueness in the even $L^p$-Minkowski problem in the subcritical range $-n < p < 0$, as well as some rigidity results for the critical exponent $p=-n$ and in the supercritical regime $p < -n$ (the sub/super prefix refers to the growth rate of the non-linearity $h_K^{1-p}$ in (\ref{eq:intro-Lp-Minkowski})). 
Let $\bar{\K}$ denote the collection of convex bodies in $\R^n$ containing the origin (possibly as a boundary point). A typical variational method for establishing the existence of $K \in \bar{\K}$ solving the $L^p$-Minkowski problem (\ref{eq:intro-Lp-Minkowski}) for a given measure $\mu$, is to minimize a natural functional $F_{\mu,p}$ having  (\ref{eq:intro-Lp-Minkowski}) as its Euler-Lagrange equation \cite{ChouWang-LpMinkowski,BLYZ-logMinkowskiProblem,BBCY-SubcriticalLpMinkowski,Kolesnikov-OTOnSphere}; for general $K \in \bar{\K}$ it is actually imperative to incorporate a maximization over all possible translations of $K$ (so that the origin remains in $K$) in the following definition, but when $\mu$ is even and the sought-after $K$ is origin-symmetric, the functional simplifies to:
\begin{equation}  \label{eq:intro-F}
\K_e \ni K \mapsto F_{\mu,p}(K) := \frac{ \frac{1}{p} \int h_K^{p} d\mu}{ V(K)^{p/n} } . 
\end{equation}
Note that convexity and origin-symmetry automatically imply that the origin lies in the interior of $K$, so that $h_K > 0$ on $S^{n-1}$, thereby greatly simplifying various arguments. 

When $\mu$ has some minimal regularity -- for instance, if:
\begin{equation} \label{eq:intro-mu-cond}
\mu = f \Leb ~,~ 0 < c \leq f \leq C ,
\end{equation}
then for any $p \in (-n,1)$ one can use the Blaschke--Santal\'o inequality to ensure that a minimum of $F_{\mu,p}$ is indeed attained -- see \cite[Section 5]{ChouWang-LpMinkowski} for details (and also \cite{BBCY-SubcriticalLpMinkowski} where it is shown that any non-negative $f \in L^{\frac{n}{n+p}}(S^{n-1})$ will actually suffice). 
Note that when $p < 0$ the coefficient $\frac{1}{p}$ in (\ref{eq:intro-F}) actually turns this into a maximization problem of a positive quantity, and that when $p=0$ as in \cite{BLYZ-logMinkowskiProblem} the above functional should be interpreted in the limiting sense. Also note that (\ref{eq:intro-mu-cond}) is clearly satisfied if $\mu = S_p K$ for some $K \in \K^2_{+,e}$. 
 The case $p=-n$ is the critical exponent for this variational argument, since then, as expounded in \cite[Section 7]{ChouWang-LpMinkowski},  (\ref{eq:intro-Lp-Minkowski}) becomes the Minkowski problem for the centro-affine Gauss-curvature, and enjoys a centro-affine equivariance which precludes using compactness arguments. Note that when $\mu$ is a singular measure, the Blaschke--Santal\'o inequality does not suffice to ensure compactness, and so additional delicate justification is required in \cite{BLYZ-logMinkowskiProblem,ChenLiZhu-LpMongeAmpere,ChenLiZhu-logMinkowski} to handle such measures in the range $p\in [0,1)$. 
  
\medskip

Once a global minimum point $K$ of $F_{\mu,p}$ has been found, it remains to show that it solves the corresponding $L^p$-Minkowski problem. 
The following was established in \cite[Lemma 4.1]{BLYZ-logMinkowskiProblem} (the result was formulated for $p=0$ and global minima but applies to all $p$ and local minima):
\begin{prop}[B\"or\"oczky--Lutwak--Yang--Zhang] \label{prop:intro-EL}
For any $p \in \R$ and non-zero finite even Borel measure $\mu$ on $S^{n-1}$, if $K \in \K_e$ is a local minimum point of $F_{\mu,p}$ then $S_p K = c \cdot \mu$ for some $c > 0$. 
\end{prop}

All references to local neighborhoods in this subsection are with respect to the Hausdorff metric in $\K_e$. 
The appearance of the constant $c > 0$ is natural since $F_{\mu,p}$ is $0$-homogeneous with respect to scaling of $K$. 
It is tempting to think that Proposition \ref{prop:intro-EL} applies to any critical point $K$ of $F_{\mu,p}$, and this is indeed true for $K \in \K^2_{+,e}$, but false in general -- see Theorem \ref{thm:intro-super-critical} (4) below. The reason is a complication which arises since a $C_e$-variation $h$ of $h_K$ may not be the support function of any convex body, and so the authors of \cite{BLYZ-logMinkowskiProblem} modified $F_{\mu,p}$ (potentially increasing it) by incorporating the associated Alexandrov body, the largest convex body whose support function is upper-bounded by $h$.  Another complication is that it does not seem possible to extend this analysis to second variations. Instead, we consider a critical point $K$ which is assumed to be in $\K^2_{+,e}$, and perform a $C^2_e$-variation of $h_K$.  Denoting by $\delta^1_K F_{\mu,p}$ and $\delta^2_K F_{\mu,p}$ the first and second $C^2_e$-variations at $K$, respectively, we observe:
\begin{prop} \label{prop:intro-local-minimum}
Let $K \in \K^{2}_{+,e}$, $p \in \R$, and let $\mu$ be any non-zero finite even Borel measure on $S^{n-1}$.
\begin{enumerate}
\item $\delta^1_K F_{\mu,p} \equiv 0$ if and only if $S_p K = c \cdot \mu$ for some $c > 0$. 
\item $\delta^2_K F_{S_p K , p} \geq 0$ if and only if $\lambda_{1,e}(-\Delta_K) \geq n-p$.
\item It is never true that $\delta^2_K F_{S_p K , p} \leq 0$. Consequently, $F_{\mu,p}$ can never have a local maximum which is in $\K^2_{+,e}$. 
\end{enumerate}
\end{prop}

It follows that whenever $\lambda_{1,e}(-\Delta_K) < n-p$, $F_{\mu,p}$ cannot attain a local minimum at $K \in \K^2_{+,e}$. This can immediately be used to deduce non-uniqueness results: 

\begin{cor} \label{cor:intro-local-minimum}
Fix $p \in (-n,1)$. Let $K_1 \in \K^2_{+,e}$ with $\lambda_{1,e}(-\Delta_{K_1}) < n-p$, and let $K_2 \in \K_e$ be any local minimum point of $F_{S_p K_1, p}$ (recall that a global minimum point always exists). Then $S_p K_1 = S_p K_2$ and yet $K_1 \neq K_2$. 
\end{cor}

Recall that $\lambda_{1,e}(-\Delta_{K_1}) > \lambda_1(-\Delta_{K_1}) = n-1$, so the assumption that $\lambda_{1,e}(-\Delta_{K_1}) < n-p$ is vacuous for $p \geq 1$ and we have therefore excluded this range from our formulation. In fact, it is conjectured that $\lambda_{1,e}(-\Delta_K) > n$, and so the expected range of $p$'s where the assumption is non-vacuous is actually $(-n,0)$. Thanks to Theorem \ref{thm:main1}, we know that unless $K_1$ above is an ellipsoid, the assumption will hold for some $p > -n$, and we obtain a non-uniqueness result in the even $L^p$-Minkowski problem for $\mu = S_p K_1$. 

\begin{thm} \label{thm:non-unique1}
For any $K_1 \in \K^2_{+,e}$ which is not an ellipsoid, there exists $q = q(K_1) \in (-n,1)$, so that for all $p \in (-n,q)$ one can find $K_2 = K_2(p) \in \K_e$ different from $K_1$ with $S_p K_1 = S_p K_2$.  
\end{thm}
\noindent
Using that $\underline{\lambda}_{1,e}(Q^n) = n$, it follows that $q(K_1)$ can be chosen to be arbitrarily close to $0$:
\begin{thm} \label{thm:non-unique2}
For any $q \in (-n,0)$, there exists a single $K_q \in \K^2_{+,e}$, so that no uniqueness holds in the even $L^p$-Minkowski problem (\ref{eq:intro-Lp-Minkowski}) for $\mu = S_p K_q$, for all $p \in (-n,q)$.
\end{thm}

It is important to stress that non-uniqueness results for the even $L^p$-Minkowski problem in the range $p < 0$ (both in the subcritical and supercritical regimes) \cite{Andrews-ClassificationOfLimitingShapesOfIsotropicCurveFlows,ChouWang-LpMinkowski,HeLiWang-MultipleSupercriticalLpMinkowski,JianLuWang-NonUniquenessInSubcriticalLpMinkowski,Li-NonUniquenessInCriticalLpMinkowskiProblem,LLL-NonUniquenessInDualLpMinkowskiProblem}
 and also in the non-even case for $p \in [0,1)$ \cite{ChenLiZhu-LpMongeAmpere,ChenLiZhu-logMinkowski,Stancu-UniquenessInDiscretePlanarL0Minkowski} have been previously obtained by many authors. 
However, the fact that \emph{any} non-ellipsoid $K_1 \in \K^{2}_{+,e}$ can be used as a witness to non-uniqueness for $p > -n$ sufficiently close to $-n$, and that furthermore one can use \emph{the same} $K_q$ as a witness for all $p \in (-n,q)$ for any $q<0$ \emph{arbitrarily close to $0$}, appears to be new and of interest. 

In the critical case $p=-n$, it is classical that there is no uniqueness even for $\mu = \Leb$ due to the centro-affine equivariance of $S_{-n}$; in particular,
the centro-affine Gauss-curvature of any (centered) ellipsoid $\EE$ in $\R^n$ is constant \cite{Tzitzeica1908,ChouWang-LpMinkowski}: 
\begin{equation} \label{eq:ellipsoids-curvature}
\exists c > 0 \;\; \;  S_{-n} \EE = c \Leb .
 \end{equation}
 It was shown by Calabi \cite{Calabi-CompleteAffineHyperspheres} that, in fact, (centered) ellipsoids are the only complete elliptic solutions to  (\ref{eq:ellipsoids-curvature}). 
  On the other hand, when $-n < p < 1$, a (centered) Euclidean ball is the unique solution $K \in \K$ to the equation $S_p K = c \Leb$, 
  as established by Brendle--Choi--Daskalopoulos -- see \cite{BCD-PowerOfGaussCurvatureFlow} and the references therein. The latter result may be extended to show that for any centered ellipsoid $\EE$ and $-n < p < 1$, $K = \EE$ is the unique solution in $\K_e$ to the equation $S_p K = S_p \EE$ -- see \cite{EMilman-IsomorphicLogMinkowski}. 
Theorem \ref{thm:non-unique1} may therefore be interpreted as a converse to the result of \cite{BCD-PowerOfGaussCurvatureFlow} (and its extension to centered ellipsoids) -- together, they show that ellipsoids are characterized as the \emph{only} members of $\K^2_{+,e}$ for which uniqueness holds in the even $L^p$-Minkowski problem in the \emph{entire} subcritical range $p \in (-n,1)$.

\medskip
One additional application we will show in Section \ref{sec:LpMinkowski} is a non-compactness result near the critical exponent $p=-n$:
\begin{prop} \label{prop:intro-diam}
Let $\mu = f \Leb$ with non-constant positive even density $f \in C^{\alpha}_e(S^{n-1})$. Given $p \in (-n,0)$, let $K_p \in \K_e$ be any local minimum point of $F_{\mu,p}$ (in particular, $S_p K_p = c_p \cdot \mu$). 
Then for any sequence $p_i \searrow -n$ we necessarily have:
\[
\lim_{i \rightarrow \infty} d_G(K_{p_i}, B_2^n)  = +\infty. 
\]
\end{prop}

Here $d_G(K,B_2^n)$ denotes the geometric distance between $K$ and $B_2^n$, defined as the ratio between the out- and in- radii of $K$, namely $\inf \{ a b > 0 \; ; \; \frac{1}{b} B_2^n \subset K \subset a B_2^n \}$. Clearly, the assumption that $f$ is non-constant is crucial for the validity of the claim, since otherwise we could take all $K_p$'s to be the Euclidean ball $B_2^n$, which is easily seen to be global minimum point of $F_{\Leb,p}$ in $\K_e$ by the Blaschke--Santal\'o and Jensen inequalities. This demonstrates a very strong rigidity property of the critical exponent $p=-n$. 

\medskip

In the supercritical regime we observe the following additional rigidity properties:
\begin{thm} \label{thm:intro-super-critical}
Let $p \leq -n$ and let $\mu$ denote a non-zero finite even Borel measure on $S^{n-1}$. 
\begin{enumerate}
\item
$F_{\mu , p}$ has no local maximum points which are in $\K^2_{+,e}$. 
\item 
$F_{\mu , p}$ has no local minimum points which are in $\K^2_{+,e}$, unless $p=-n$ and $\mu = c \Leb$. 
\item 
If $\mu = f \Leb$ with positive even density $f \in C^{\alpha}_e(S^{n-1})$, then $F_{\mu , p}$ has no local minima at all, unless $p=-n$ and $\mu = c \Leb$. 
\item 
If $p < -n$ and $\mu = f \Leb$ with an even density satisfying $\norm{f}_{L^{\frac{n}{n+q}}(\Leb)} \in (0,\infty)$ for some $q \in (p,-n)$, then $-F_{\mu,p}(K)$ is coercive under a volume constraint, i.e. tends to $+\infty$ uniformly as $d_G(K,B_2^n) \rightarrow \infty$ if $V(K)$ is fixed. Consequently, $\K_e \ni K \mapsto F_{\mu , p}(K)$ has no global minimum but attains a global maximum at a point $K_{\max} \in \K_e \setminus \K^{2}_{+,e}$. If moreover $\mu$ satisfies (3) above, then necessarily $S_p K_{\max} \neq c \cdot \mu$. 
\end{enumerate}
\end{thm}
In particular, this means that if one wishes to use a variational approach involving $F_{\mu,p}$ for establishing existence in the supercritical regime of the even $L^p$-Minkowski problem with a measure $\mu = f \Leb$ having positive even density $f \in C^{\alpha}_e(S^{n-1})$, one necessarily needs to resort to saddle-point methods, or alternatively, to restrict to a certain subclass of symmetric convex bodies as in \cite{LuWang-CriticalAndSupercriticalLpMinkowski,JianLuZhu-UnconditionalCriticalLpMinkowski} -- see Remark \ref{rem:symmetries}. 

\subsection{The non-smooth case}

It is also interesting to try and extend Theorems \ref{thm:main1} and \ref{thm:main-direction} to the general (possibly non-smooth) case. When $K \in \K_e \setminus \K^2_{+,e}$, the definition of $\Delta_K$ as a regular differential operator does not make sense, and so there are several natural options for replacing the spectral parameter $\lambda_{1,e}(-\Delta_K)$:
\begin{equation} \label{eq:intro-lambda-inqs}
\underline{\lambda}_{1,e}(K) \leq \overline{\lambda}_{1,e}(K) \leq \lambdaC_{1,e}(K) .
\end{equation}
The first two variants above are defined as:
\begin{align*}
\underline{\lambda}_{1,e}(K) & := \liminf_{\K^2_{+,e} \ni K_i \rightarrow K \text{ in $C$}}  \lambda_{1,e}(-\Delta_{K_i}) , \\
\overline{\lambda}_{1,e}(K) & := \limsup_{\K^2_{+,e} \ni K_i \rightarrow K \text{ in $C$}}  \lambda_{1,e}(-\Delta_{K_i}) .
\end{align*}
Here $C$-convergence is synonymous with convergence in the Hausdorff metric -- see Section \ref{sec:prelim} for additional details and notation. The smallest variant $\underline{\lambda}_{1,e}$ is the natural one when considering lower-bounds, and was therefore used in \cite{KolesnikovEMilman-LocalLpBM}. The largest variant $\lambdaC_{1,e}$  is the natural one for studying upper-bounds as in this work, and so we will use it. Following the idea of Putterman \cite{Putterman-LocalToGlobalForLpBM} of using mixed-volumes of convex bodies $V(L[m],K[n-m])$ instead of second-order differentiation of $C^2$ test-functions, we define:
\begin{equation} \label{eq:def-lambdaC}
\lambdaC_{1,e}(K) := (n-1) \inf_{L \in \K_e} \set{ \frac{ \int (\frac{h_L}{h_K})^2 dV_K - V(L[2],K[n-2]) }{ \int (\frac{h_L}{h_K})^2 dV_K - \frac{V(L[1],K[n-1])^2}{V(K)} } \; ; \; \frac{h_L}{h_K}  \text{ non-constant $V_K$-a.e.} }.
\end{equation}
Note that the Cauchy-Schwarz inequality and Minkowski's second inequality imply that both numerator and denominator are non-negative. It follows from \cite[Appendix]{Putterman-LocalToGlobalForLpBM} that:
\[
\lambdaC_{1,e}(K) = \lambda_{1,e}(-\Delta_K) \;\;\; \forall K \in \K^2_+,
\]
 and  since $\lambdaC_{1,e}(K)$ is clearly upper semi-continuous with respect to $C$-convergence (being the best constant in an inequality relating two non-negative $C$-continuous functions of $K \in \K_e$), the second inequality in (\ref{eq:intro-lambda-inqs}) immediately follows.

\begin{thm} \label{thm:main2}
For all $K \in \K_e$, we have:
\[
\lambdaC_{1,e}(K) \leq 2 n 
\]
(with equality for all centered ellipsoids). If equality holds above then $K$ is a linear image of a tangential body to a (centered) Euclidean ball.
\end{thm}

A convex body $K \subset \R^n$ is called a tangential body to a (centered) Euclidean ball if:
\begin{equation} \label{eq:intro-tangential-body}
K = \bigcap_{\theta \in S} \set{ x \in \R^n \; ; \; \scalar{x,\theta} \leq R } 
\end{equation}
for some $R > 0$ and closed subset $S \subset S^{n-1}$ (which is necessarily not contained in any hemisphere since $K$ is compact). It is known (see \cite[Theorem 2.2.10 and pp. 149,386]{Schneider-Book-2ndEd}) that $K$ is a tangential body to a (centered) Euclidean ball if and only if $h_K(\theta) = R$ for $S_K$-a.e. $\theta \in S^{n-1}$, and in that case $S$ coincides with the support of $S_K$. It follows that if $S_K$ has full support then $K$ must be a Euclidean ball itself, thereby slightly relaxing the assumption that $K \in \K^2_{+,e}$ in the equality case characterization of Theorem \ref{thm:main1}. 

\begin{rem}
If $\lambdaC_{1,e}(K) = 2 n$, we can in fact show that $K$ satisfies several additional properties -- see Remark \ref{rem:tangential} -- but we do not know how to deduce that $K$ must be an ellipsoid. The reason is that the assumption that $S_K$ has full support is crucially used in our characterization of the equality case in Theorem \ref{thm:main-direction}, and we do not know whether the result still holds for general $K \in \K$ -- see the more general Theorem \ref{thm:directions} below and Remark \ref{rem:L4-cosine}. We leave these questions as interesting open problems. 
\end{rem}

\medskip

The rest of this work is organized as follows. In Section \ref{sec:prelim} we recall some necessary preliminaries. In Section \ref{sec:calc} we calculate various expressions involving powers of $\lin_{K,\xi}$. In Section \ref{sec:directions} we prove an extended version of Theorem \ref{thm:main-direction}. In Section \ref{sec:main} we provide proofs of Theorems \ref{thm:main1} and \ref{thm:main2}. In Section \ref{sec:LpMinkowski} we conclude by providing proofs of our various applications to the even $L^p$-Minkowski problem. 

\medskip

\noindent{\textbf{Acknowledgments}.} 
I thank K\'aroly J. B\"or\"oczky and Gaoyong Zhang for their comments regarding an earlier version of this manuscript. I also thank the referees for their useful comments and suggestions.

\section{Preliminaries} \label{sec:prelim}

We begin with some preliminaries, referring to \cite{Schneider-Book-2ndEd,KolesnikovEMilman-LocalLpBM} and the references therein for additional information. 

\subsection{Notation}

We work in Euclidean space $(\R^n,\scalar{\cdot,\cdot})$, denoting $|x| = \sqrt{\scalar{x,x}}$.  
A convex body in $\Real^n$ is a convex, compact set with non-empty interior. We denote by $\K$ the collection of convex bodies in $\Real^n$ having the origin in their interior.  The support function $h_K : \Real^n \rightarrow \R_+$ of $K \in \K$ is defined as: 
\[
h_K(x^*) := \max_{x \in K} \scalar{x^*,x} \; ~,~ x^* \in \Real^n . 
\]
It is easy to see that $h_K$ is continuous, convex and positive outside the origin. Clearly, it is $1$-homogeneous, so we will mostly consider its restriction to the Euclidean unit-sphere $S^{n-1}$. Conversely, a convex $1$-homogeneous function $h : \Real^n \rightarrow \R_+$ which is positive outside the origin is necessarily the support function of some $K \in \K$. The polar-body $K^{\circ}$ of $K \in \K$ is defined as the level-set $\{ h_K \leq 1 \}$; duality implies that $(K^{\circ})^{\circ} = K$. 
The Minkowski gauge function of $K \in \K$ is defined as:
\[
\norm{x}_K := \inf \{ t > 0 \; ; \; x \in t K \} . 
\]
Note that $h_K = \norm{\cdot}_{K^{\circ}}$ and $h_{K^{\circ}} = \norm{\cdot}_K$. 

\medskip

We denote by $C^k(S^{n-1})$ and $C^{k,\alpha}(S^{n-1})$, $k = 0,1,2,\ldots$ and $\alpha \in (0,1)$, the space of $k$-times continuously and $\alpha$-H\"older differentiable functions on $S^{n-1}$, respectively, equipped with their usual corresponding topologies. When $k=0$, we simply write $C(S^{n-1})$ and $C^{\alpha}(S^{n-1})$. It is known \cite[Section 1.8]{Schneider-Book-2ndEd} that convergence of elements of $\K$ in the Hausdorff metric is equivalent to convergence of the corresponding support functions in the $C(S^{n-1})$ norm; we will refer to this as $C$-convergence for brevity. We denote by $C^k_{>0}(S^{n-1})$ the convex cone of positive functions in $C^k(S^{n-1})$. The subset of support functions of convex bodies in $\K$ is denoted by $C^k_h(S^{n-1})$. 

\medskip

Let  $\nabla_{S^{n-1}}$ be the standard Levi--Civita connection on the sphere $S^{n-1}$ with its canonical Riemannian metric $\delta = \delta_{S^{n-1}}$. We use $\omega_i$ to denote the $1$-form $\omega$ in a local  frame $e_1,\ldots,e_{n-1}$ on $S^{n-1}$, and $(w_i)_j$ to denote the covariant derivative $\nabla_{S^{n-1}} \omega$. For a function $h \in C^2(S^{n-1})$, we use $h_i$ and $h_{ij}$ to denote $\nabla_{S^{n-1}} h = dh$ and $\nabla^2_{S^{n-1}} h$ in this frame, respectively, 
e.g. $(h_i)_j = h_{ij}$. Extending $h$ to a $1$-homogeneous function on $\Real^n$ and denoting by $\nabla_{\Real^n}$ the covariant derivative on Euclidean space $\Real^n$, we define the symmetric $2$-tensor $D^2 h$ on $S^{n-1}$ as the restriction of $\nabla^2_{\Real^n} h$ onto $T S^{n-1}$; in our local frame, this reads as:
\[
D^2_{ij} h = \nabla_{\Real^n}^2 h(e_i,e_j) = h_{ij} + h \delta_{ij} ~,~ i,j = 1,\ldots,n-1 .
\]
Note that $h \in C^2_{>0}(S^{n-1})$ is a support-function of $K \in \K$ if and only if $D^2 h_K \geq 0$. 
\medskip

We denote by $\K^m_+$ the subset of $\K$ of convex bodies with $C^m$ boundary and strictly positive curvature. By \cite[pp.~115-116,120-121]{Schneider-Book-2ndEd}, for $m \geq 2$, $K \in \K^m_+$ if and only if $h_K \in C^m(S^{n-1})$ and $D^2 h_K > 0$. It is well-known that $\K^2_+$ is dense in $\K$ in the $C$-topology \cite[p. 185]{Schneider-Book-2ndEd}.

\medskip

A convex body $K$ is called origin-symmetric if $K = -K$. We will always use $S_e$ to denote the origin-symmetric (or even) members of a set $S$, e.g. $\K_e$ and $\K^2_{+,e}$ denote the subsets of origin-symmetric bodies in $\K$ and $\K^2_+$, respectively, and $C^2_{e}(S^{n-1})$ and $C^2_{h,e}(S^{n-1})$ denote the subsets of even functions in $C^2(S^{n-1})$ and $C^2_h(S^{n-1})$, respectively. 

\medskip

$GL_n$ denotes the group of non-singular linear transformations in $\R^n$, and $SL_n$ denotes the subgroup of volume and orientation preserving elements. We use the terms ``centro-affine" and ``linear" interchangeably. From here on, all Euclidean balls and ellipsoids are assumed to be centered at the origin.

\subsection{Brunn--Minkowski theory}

Given a convex body $K \subset \R^n$, its surface-area measure $S_K$ is defined as the push-forward under the Gauss map $\n_{\partial K} : \partial K \rightarrow S^{n-1}$ of $\H^{n-1}|_{\partial K}$. Here $\n_{\partial K}$ denotes the outer unit-normal to $K$ and $\H^{n-1}$ is the $(n-1)$-dimensional Hausdorff measure. 
When $K \in \K^2_{+}$, we have:
\[
S_K = \det(D^2 h_K) \Leb . \]
More generally, Lutwak introduced in \cite{Lutwak-Firey-Sums} the $L^p$ surface-area measure of $K$ as:
\[
S_p(K,\cdot) = S_p K := h^{1-p}_K S_K . 
\]
The cone-volume measure $V_K$ on $S^{n-1}$ is defined as:
\[
V_K := \frac{1}{n} h_K S_K ;
\]
it is obtained by first pushing forward the Lebesgue measure on $K$ via the cone-map $K \ni x \mapsto x / \norm{x}_K \in \partial K$, and then pushing forward the resulting cone-measure on $\partial K$ via the Gauss map $\n_{\partial K} : \partial K \rightarrow S^{n-1}$. 
It is known that all of the measures above are weakly continuous (i.e. in duality with $C(S^{n-1})$) with respect to $C$-convergence of $K$ \cite[pp. 212-215]{Schneider-Book-2ndEd}.

The Blaschke--Santal\'o inequality states that for all $K \in \K$:
\[
\min_{x \in \text{int}(K)} V(K) V((K-x)^{\circ}) \leq V(B_2^n)^2 . 
\]
When $K \in \K_e$, the (unique) minimum above is attained at $x=0$.

The Minkowski sum $K_1 + K_2$ of two sets is defined as $\{ x_1 + x_2 \; ; \; x_i \in K_i \}$. Note that this operation is additive on the level of support-functions: $h_{K_1 + K_2} = h_{K_1} + h_{K_2}$. It was shown by Minkowski that when $\set{K_i}_{i=1}^m$ are convex bodies in $\Real^n$, then the volume of their Minkowski sum is a polynomial with non-negative coefficients in the scaling parameters:
\[
V(\sum_{i=1}^m t_i K_i) = \sum_{1 \leq i_1,\ldots,i_n \leq m} t_{i_1} \cdot \ldots \cdot t_{i_n} V(K_{i_1},\ldots,K_{i_n}) \;\;\; \forall t_i \geq 0 .
\]
The coefficient $V(K_{i_1},\ldots,K_{i_n}) \geq 0$ is called the mixed volume of the $n$-tuple $(K_{i_1},\ldots,K_{i_n})$; it is clearly multi-linear in its arguments (with respect to Minkowski addition), and uniquely defined by requiring that it be invariant under permutations. Moreover, the mixed volume is continuous with respect to  (joint) convergence of its arguments in the Hausdorff metric \cite[Section 5.1]{Schneider-Book-2ndEd}.

Using multi-linearity, one can extend the definition of mixed volume to a $n$-tuple of functions $(h_1,\ldots,h_n)$, each of which is in $C_{h-h}(S^{n-1}) := C_h(S^{n-1}) - C_h(S^{n-1})$, i.e. is the difference of two support functions \cite[Section 5.2]{Schneider-Book-2ndEd}. Note that any $h \in C^2(S^{n-1})$ is clearly the difference of two support functions, so the above extension applies in particular to the case that $h_i \in C^2(S^{n-1})$. In that case, one can write  (see e.g. \cite[Section 4]{KolesnikovEMilman-LocalLpBM}) an explicit formula for the mixed volume involving the mixed determinant of the second derivatives $D^2 h_i$. In any case, we emphasize that $V$ remains multi-linear and invariant under permutation of its arguments.

Given convex bodies $K,L \in \K$, we will use the abbreviation:
\[
V(L[m],K[n-m]) = V(\underbrace{L,\ldots,L}_{\text{$m$ times}},  \underbrace{K,\ldots,K}_{\text{$n-m$ times}}) . 
\]
Given functions $f_1,\ldots,f_n$ so that $f_i h_K \in C_{h-h}(S^{n-1})$ are the difference of support functions, we will also use:
\begin{align}
\nonumber V_K(f_1,\ldots,f_n) & := V(f_1 h_K , \ldots, f_n h_K) ,\\
\nonumber V_K(f_1,\ldots,f_m) & := V_K(f_1 ,\ldots, f_m , \underbrace{1 , \ldots , 1}_{\text{$n-m$ times}}) , \\
\label{eq:semicolon} V_K(f_1; m) & := V_K(\overbrace{f_1,\ldots,f_1}^{\text{$m$ times}}) . \end{align}

Note that we always have:
\begin{equation} \label{eq:V1}
V_K(f_1 ; 1) = \frac{1}{n} \int_{S^{n-1}} f_1 h_K dS_K = \int_{S^{n-1}} f_1 dV_K ,
\end{equation}
and that by Minkowski's second inequality \cite[Section 7.2]{Schneider-Book-2ndEd}:
\begin{equation} \label{eq:Minkowski2}
V(L[1],K[n-1])^2 \geq V(L[2],K[n-2]) V(K) ,
\end{equation}
which is equivalent to the celebrated Brunn--Minkowski inequality \cite[Section 7.1]{Schneider-Book-2ndEd}:
\begin{equation} \label{eq:BM-BM}
V(K+L)^{\frac{1}{n}} \geq V(K)^{\frac{1}{n}} + V(L)^{\frac{1}{n}} . 
\end{equation}

\subsection{Hilbert--Brunn--Minkowski operator}

Fix $K \in \K^{2}_{+}$, and introduce the following Riemannian (positive-definite) metric on $S^{n-1}$:
\[
g_K := \frac{D^2 h_K}{h_K} > 0 . 
\]
As customary, we will use $g_K^{ij}$ to denote the inverse metric in a local frame. The Levi-Civita connection associated to $g_K$ is denoted by $\nabla$; there should be no confusion with $\nabla_{S^{n-1}}$, since we will only use it to write:
\[
g_K(\nabla z, \nabla w) = g_K^{ij} z_i w_j ~,~ |\nabla z|_{g_K}^2 = g_K^{ij} z_i z_j . 
\]
As usual, Einstein summation convention of summing over repeated indices will be freely employed. 

The Hilbert--Brunn--Minkowski operator $\Delta_K : C^2(S^{n-1}) \rightarrow C(S^{n-1})$ is defined as:
\[
\Delta_K z := g_K^{ij} \frac{ D^2_{ij}(z h_K) - z D^2_{ij} h_K}{h_K} =  g_K^{ij} ( z_{ij} + (\log h_K)_i z_j + (\log h_K)_j z_i) = 
g_K^{ij} \frac{ (h_K^2 z_i)_j }{h_K^2} . 
\]
Note that we are using a slightly different normalization than in our previous work \cite{KolesnikovEMilman-LocalLpBM}, where the Hilbert--Brunn--Minkowski operator (denoted $L_K$) was defined as $L_K := \frac{1}{n-1} \Delta_K$. 
Clearly, $\Delta_K$ is an elliptic second-order differential operator with vanishing zeroth order term, and in particular $\Delta_K 1 = 0$. This operator was introduced by Hilbert (under somewhat different normalization and gauge) in his proof of the Brunn--Minkowski inequality (see \cite[Section 52]{BonnesenFenchelBook}). 
Very recently in \cite{EMilman-IsomorphicLogMinkowski}, we interpreted $g_K$ and $\Delta_K$ as the centro-affine metric and Laplacian of $\partial K$, respectively (obtained by equipping $\partial K$ with the centro-affine normalization and parametrizing on $S^{n-1}$ via the Gauss map), but we will not require this here. 

It is known that 
for all $z,w \in C^2(S^{n-1})$:
\begin{equation} \label{eq:Delta-V2}
\frac{1}{n-1} \int_{S^{n-1}} (\Delta_K z) w  dV_K = V_K(w,z) - \int_{S^{n-1}} w z dV_K ,
\end{equation}
and that the following integration-by-parts formula holds:
\[
 \int_{S^{n-1}} (-\Delta_K z) w dV_K =  \int_{S^{n-1}} g_K(\nabla z , \nabla w) dV_K = \int_{S^{n-1}} (-\Delta_K w) z dV_K  \]
(see \cite[(5.5)-(5.6)]{KolesnikovEMilman-LocalLpBM} for the case $z=w$; the general case follows by polarization). 
Consequently, we may interpret $\Delta_K$ as the weighted Laplacian on the weighted Riemannian manifold $(S^{n-1}, g_K, V_K)$ (see e.g. \cite{KolesnikovEMilmanReillyPart1,KolesnikovEMilmanReillyPart2}).  
It follows that  $-\Delta_K$ is a symmetric positive semi-definite operator on $L^2(V_K)$. It uniquely extends to a self-adjoint positive semi-definite operator with Sobolev domain $H^2(S^{n-1})$ and compact resolvent, which we continue to denote by $-\Delta_K$. Its (discrete) spectrum is denoted by $\sigma(-\Delta_K)$. It was shown by Hilbert that Minkowski's inequality (\ref{eq:Minkowski2}) is equivalent to the statement that $\lambda_1(-\Delta_K) \geq n-1$, where $\lambda_1$ denotes the first non-zero eigenvalue; Hilbert confirmed that $\lambda_1(-\Delta_K) = n-1$, thereby obtaining a spectral proof of (\ref{eq:Minkowski2}) and thus the Brunn--Minkowski inequality (\ref{eq:BM-BM}).  

\medskip

Denote by $H^2_e(S^{n-1})$ the even elements of the Sobolev space $H^2(S^{n-1})$ and by $\mathbf{1}^{\perp}$ those elements $f$ for which $\int f dV_K = 0$. The first non-trivial \emph{even} eigenvalue of $-\Delta_K$ is defined as:
\begin{align}
\nonumber \lambda_{1,e}(-\Delta_K) & := \min \sigma(-\Delta_K|_{H^2_e(S^{n-1}) \cap \mathbf{1}^{\perp}}) \\
\label{eq:RR} & = \inf \set{ \frac{\int_{S^{n-1}}|\nabla z|_{g_K}^2 dV_K}{\Var_{V_K}(z)} \; ; \; \text{non-constant } z \in C^2_{e}(S^{n-1})  } ,
\end{align}
where:
\[
\Var_{V_K}(z) := \int_{S^{n-1}} z^2 dV_K - \frac{(\int_{S^{n-1}} z dV_K)^2}{V(K)}.
\]  

\subsection{$K$-adapted linear functionals}

Let $K \in \K$. We introduce the following two vector-fields on $\partial K$, which are well-defined for $\H^{n-1}$-a.e. $x \in \partial K$:
\[
x^*  := \nabla_{\R^n} \norm{x}_K  ~,~ \theta^* := \frac{x^*}{|x^*|} . 
\]
Observe that $\theta^* : \partial K \rightarrow S^{n-1}$ precisely coincides with the Gauss map $\n_{\partial K}$. 
It easily follows that $|x^*| = 1/ h_K(\theta^*)$ and thus $x^* \in \partial K^{\circ}$ and  $\scalar{x,x^*} = 1$ for $\H^{n-1}$-almost-every $x \in \partial K$. We naturally extend $x^*$ and $\theta^*$ as $0$-homogeneous vector-fields, defined $\H^n$-a.e. on the entire $\R^n$. Note that $V_K$ is then exactly the push-forward of $\H^n|_{K}$ via $\theta^*$. 

We will typically consider $x^*$ as a function of $\theta^*$:
\[
x^* : S^{n-1} \ni \theta^* \mapsto \frac{\theta^*}{h_K(\theta^*)} . 
\]
The associated ``$K$-adapted linear" functional on $S^{n-1}$ in the direction of $\xi \in \R^n$ is defined as:
\[
\lin_{K,\xi}(\theta^*) := \scalar{x^*,\xi} = \frac{\scalar{\theta^*,\xi}}{h_K(\theta^*)} . 
\]
It is immediate to verify that:
\[
-\Delta_K \lin_{K,\xi} = (n-1) \lin_{K,\xi} \;\;\; \forall \xi \in \R^n . 
\]
In his proof that $\lambda_1(-\Delta_K) = n-1$, it was in fact shown by Hilbert (see \cite[p. 110]{BonnesenFenchelBook}) that the corresponding eigenspace is precisely $n$-dimensional, i.e. that there are no other eigenfunctions for the eigenvalue $n-1$ besides $\{ \lin_{K,\xi} \}_{\xi \in \R^n}$.

Since in this work we will not use the dual representation $\theta$ to that given by $\theta^*$ (namely $\theta = \frac{x}{|x|}$), we will often write $\theta$ instead of $\theta^*$ in our various calculations below.

\subsection{$\Gamma_{-p}$-bodies and $S_2$-isotropic position}

In \cite{LYZ-LpJohnEllipsoids}, Lutwak--Yang--Zhang introduced the convex bodies $\Gamma_{-p} K$ for $p \geq 1$, obtained via the $L^p$ spherical cosine transform of the $L^p$ surface-area measure $S_p K$:
\begin{equation} \label{eq:Gamma-p}
\norm{x}^p_{\Gamma_{-p} K} := \frac{1}{V(K)} \int_{S^{n-1}} \abs{\scalar{x,\theta}}^p dS_p(K,\theta) . 
\end{equation}
Up to normalization, they coincide with the polar $L^p$ projection bodies $\Pi^*_p K$ introduced in \cite{LYZ-LpAffineIsoperimetricInqs}. It is interesting to note that their gauge function also coincides (up to normalization) with the $L^p(V_K)$-norm of our $K$-adapted linear functionals:
\begin{equation} \label{eq:Gamma-p-lin}
\norm{x}^p_{\Gamma_{-p} K} = \frac{n}{V(K)} \int_{S^{n-1}} |\lin_{K,x}|^p dV_K . 
\end{equation}

When $p=2$, note that $\Gamma_{-2} K$ defines an ellipsoid, first introduced and studied in \cite{LYZ-NewEllipsoid}. In \cite{LYZ-LpJohnEllipsoids}, it was realized that $E_2 K = \Gamma_{-2} K$ is a member of an entire family of ellipsoids $E_p K$ (indexed by $p \in (0,\infty]$), called the $L^p$ John ellipsoids, which are $GL_n$-covariantly associated to $K$ (with the classical John ellipsoid corresponding to $p=\infty$). These ellipsoids are characterized by the property that the total $L^p$ surface-area $\norm{S_p K}$ is minimized among all $SL_n$ images of $K$ if and only if $E_p K$ is a Euclidean ball. 

Observe from (\ref{eq:Gamma-p}) that $E_2 K = \Gamma_{-2} K$ is a Euclidean ball if and only if the measure $S_2 K$ is isotropic.  
Recall that a measure $\mu$ on $S^{n-1}$ is called isotropic if:
\[
\int_{S^{n-1}} \scalar{x,u}^2 d\mu(x) = \frac{1}{n} |u|^2 \norm{\mu} \;\;\; \forall u \in \R^n .
\]
Note that the left-hand-side defines a quadratic form in $u$ on $\R^n$ whose trace is always $\int_{S^{n-1}} |x|^2 d\mu(x) = \norm{\mu}$, so the isotropicity condition amounts to demanding that this quadratic form is the appropriate multiple of the identity.

 It follows e.g. from \cite[Lemma $1^*$]{LYZ-NewEllipsoid} or \cite[Lemma 4.1]{LYZ-LpJohnEllipsoids}  that for any convex body $K \in \K$ there exists $T \in SL_n$ so that $S_2 T(K)$ is isotropic, and that this $T$ is unique up to composition with rotations from the left. We will say that $T(K)$ is the ``$S_2$-isotropic position" of $K$ (the definite article is well justified since we typically identify convex bodies modulo rotations).

\section{Calculations} \label{sec:calc}

\subsection{Smooth case}

Let $K \in \K^2_{+}$. Recall that $\Delta_K$ is the weighted Laplacian on $(S^{n-1},g_K,V_K)$, and therefore satisfies by the chain-rule for any  $z \in C^2(S^{n-1})$ and $\varphi \in C^2(\R)$:
\begin{equation} \label{eq:chain-rule}
\Delta_K (\varphi(z)) = \varphi'(z) \Delta_K z + \varphi''(z) |\nabla z|^2_{g_K} . 
\end{equation}

We are now ready to establish the following new observation. Note that it is very rare to precisely relate between the Dirichlet energy and the $L^2$-norm of a test-function -- a property typically reserved for eigenfunctions of the associated Laplacian. 

\begin{proposition} \label{prop:energy-lin}
Let $K \in \K^2_{+}$ and $\xi \in \R^n$. For any natural $p \in \N$:
\[
\int_{S^{n-1}} |\nabla (\lin^p_{K,\xi})|_{g_K}^2 dV_K = (n-1) \frac{p^2}{2p-1} \int_{S^{n-1}} \lin^{2p}_{K,\xi} dV_K . 
\]
\end{proposition}
\begin{proof}
Abbreviate $\lin = \lin_{K,\xi}$. Recall that $-\Delta_K(\lin) = (n-1) \lin$. 
Given $\varphi \in C^2(\R)$, denote $\varphi_1(t) := \varphi(t) \varphi'(t) t$ and $\varphi_2(t) := \varphi(t) \varphi''(t)$. 
It follows from (\ref{eq:chain-rule}) that:
\[
-\varphi(\lin) \Delta_K(\varphi(\lin)) = (n-1) \varphi_1(\lin)  - \varphi_2(\lin) |\nabla \lin|^2_{g_K} .
\]
On the other hand:
\[
|\nabla \varphi(\lin)|_{g_K}^2 =  (\varphi')^2(\lin)  |\nabla \lin|^2_{g_K} .
\]
But since:
\[
\int_{S^{n-1}} |\nabla \varphi(\lin)|_{g_K}^2 dV_K = - \int_{S^{n-1}} \varphi(\lin) \Delta_K(\varphi(\lin)) dV_K ,
\]
it follows that whenever $(\varphi')^2$ and $\varphi_2$ are proportional, we can express the latter integral as a multiple of an integral over $\varphi_1(\lin)$ only. In particular, when $\varphi(t) = t^p$, we obtain:
\[
p^2 \int_{S^{n-1}} \lin^{2p-2} |\nabla \lin|^2_{g_K}  dV_K = (n-1) p \int_{S^{n-1}} \lin^{2p} dV_K - p (p-1) \int_{S^{n-1}} \lin^{2p-2}  |\nabla \lin|^2_{g_K} dV_K ,
\]
and hence:
\[
\int_{S^{n-1}} |\nabla (\lin^p)|_{g_K}^2 dV_K  = p^2 \int_{S^{n-1}} \lin^{2p-2} |\nabla \lin|^2_{g_K}  dV_K = (n-1) p \frac{p^2}{p^2 + p (p-1)}  \int_{S^{n-1}} \lin^{2p} dV_K ,
\]
as asserted. 
\end{proof}

\subsection{General case}

Our plan will be to use $\lin_{K,\xi}^2$ as a test-function in the Rayleigh-Ritz quotient (\ref{eq:RR}) for upper-bounding $\lambda_{1,e}(-\Delta_K)$ when $K \in \K^2_{+,e}$. To handle general convex bodies $K \in \K_e$, we will also require to represent $\lin_{K,\xi}^2$ as the difference of a support function and a multiple of $h_K$ -- this is handled in the present subsection.

\begin{lemma} \label{lem:diff}
Let $K \in \K$ with $\frac{1}{R} B_2^n \subset K$ ($R > 0$). Then for any even $p \in 2 \N$ and $\xi \in S^{n-1}$, 
\[
h_{K_{R,p,\xi}} := h_K( (p-1) R^p + \lin^p_{K,\xi}) = (p-1) R^p h_K + \frac{\scalar{\cdot,\xi}^p}{h_K^{p-1}} 
\]
is the support function of a convex body $K_{R,p,\xi} \in \K$. 
\end{lemma}
\begin{proof}
The function $h_{K_{R,p,\xi}}$ is clearly $1$-homogeneous on $\R^n$, as well as continuous and positive outside the origin, so it remains to establish that it is convex. 

Assume first that $h_K$ is $C^1$ smooth. Note that the convexity of a $1$-homogeneous function $f \in C^1_{>0}(\R^n \setminus \{0\})$ is equivalent to the property that:
\begin{equation} \label{eq:convex-equiv}
\scalar{\nabla f(y),z} \leq f(z) \;\;\; \forall y,z \in \R^n 
\end{equation}
(defining $\nabla f(0) := 0$). 
Indeed, if $f$ is convex then by homogeneity $f(y + \eps z) \leq f(y) + \eps f(z)$ and (\ref{eq:convex-equiv}) holds. Conversely, if (\ref{eq:convex-equiv}) holds, subtracting Euler's identity $\scalar{\nabla f(y),y} = f(y)$, it follows that:
\[
f(y) + \scalar{\nabla f(y), z-y} \leq f(z) \;\;\; \forall y,z \in \R^n ,
\]
which is precisely the property that the graph of $f$ lies above every tangent plane, so $f$ is convex. 

We therefore verify (\ref{eq:convex-equiv}) for $h_{K_{R,p,\xi}}$. Given  $y,z \in \R^n$, our goal is to show that:
\[
 (p-1) R^p \scalar{\nabla h_K(y),z} + p \frac{\scalar{\xi,y}^{p-1}}{h^{p-1}_K(y)} \scalar{\xi,z} - (p-1) \frac{\scalar{\xi,y}^{p}}{h^{p}_K(y)} \scalar{\nabla h_K(y),z} \leq (p-1) R^p h_K(z) + \frac{\scalar{\xi,z}^p}{h^{p-1}_K(z)} . 
\]
Set $A_x := \frac{\scalar{\xi,x}}{h_K(x)}$, and note that since $\xi \in S^{n-1}$ and $K \supset \frac{1}{R} B_2^n$ we have $|A_x| \leq R$ for all $x$. Rewriting the last inequality, we would like to show that:
\[
(p-1) (R^p -  A_y^p) \scalar{\nabla h_K(y),z} \leq (p-1) R^p h_K(z) - p A_y^{p-1} \scalar{\xi,z} + h_K(z) A_z^p . 
\]
Using that $R^p - A_y^p \geq 0$ and that $\scalar{\nabla h_K(y),z} \leq h_K(z)$ by (\ref{eq:convex-equiv}), it is enough to establish that:
\[
0 \leq (p-1) A_y^p h_K(z) - p A_y^{p-1} \scalar{\xi,z} + h_K(z) A_z^p .
\]
Dividing by $p h_K(z)$, this is equivalent to:
\[
\frac{p-1}{p} A_y^p + \frac{1}{p} A_z^p - A_y^{p-1} A_z \geq 0 . 
\]
But the latter is a consequence of the arithmetic-geometric means inequality (recall that $p$ is even), and so convexity of $h_{K_{R,p,\xi}}$ is established. 

The claim has been established for $K$'s with $C^1$-smooth support-function. For general $K \in \K$, simply approximate $K$ in the $C$-topology using the former class; since the inradius assumption $\frac{1}{R} B_2^n \subset K$ is continuous in the latter topology and since convexity is preserved under pointwise convergence, the assertion follows for general $K$. 
\end{proof}

The following may be of independent interest. Recall our abbreviation $V_K(f;2) = V_K(f,f)$ from (\ref{eq:semicolon}). 

\begin{proposition} \label{prop:V2-lin}
For any $K \in \K$, even $p \in 2 \N$ and $\xi \in \R^n$, we have:
\[
V_K\brac{\lin_{K,\xi}^p ; 2} = - \frac{(p-1)^2}{2 p - 1} \int_{S^{n-1}} \lin_{K,\xi}^{2p} dV_K . 
\]
\end{proposition}
\begin{rem}
By the previous lemma $ h_K \lin_{K,\xi}^p = \scalar{\cdot,\xi}^p/h_K^{p-1}$ is the difference of two support functions, and so the mixed volume on the left is well-defined by multi-linearity. 
\end{rem}
\begin{proof}[Proof of Proposition \ref{prop:V2-lin}] 
The assertion for $K \in \K^2_+$ follows by  (\ref{eq:Delta-V2}) and Proposition \ref{prop:energy-lin}, since (abbreviating $\lin = \lin_{K,\xi}$):
\begin{align*}
& V_K(\lin^p ; 2) - \int_{S^{n-1}} \lin^{2p} dV_K  = \frac{1}{n-1} \int_{S^{n-1}} \lin^p \Delta_K (\lin^p) dV_K \\
& = -\frac{1}{n-1} \int_{S^{n-1}} |\nabla (\lin^p)|_{g_K}^2 dV_K = -\frac{p^2}{2p-1} \int_{S^{n-1}} \lin^{2p} dV_K . 
\end{align*}

For general $K \in \K$, approximate it in the $C$-topology using $K_i \in \K^2_+$. By homogeneity we may assume that $\xi \in S^{n-1}$, and let $R > 0$ be so that $\frac{1}{R} B_2^n \subset K_i$ for all $i$ (in particular, all $h_{K_i} \geq R > 0$ are uniformly bounded away from zero on $S^{n-1}$).  Recall by Lemma \ref{lem:diff} that $h_{K_i} \lin^p_{K_i,\xi} = h_{K_i,R,\xi} - (p-1) R^p h_{K_i}$ is the difference of support functions. Letting $i \rightarrow \infty$, we see that $h_{K_i}$, $\lin_{K_i,\xi}$ and hence $h_{K_i,R,\xi}$ converge in $C$-norm to $h_K$, $\lin_{K,\xi}$ and $h_{K,R,\xi}$, respectively. By continuity of mixed volumes in the $C$-topology, it follows that $V_{K_i}(\lin_{K_i,\xi}^p ; 2)$ converges to $V_K(\lin_{K,\xi}^p ; 2)$. Finally, as $V_{K_i}$ converges weakly to $V_K$, it follows that $\int_{S^{n-1}} \lin_{K_i,\xi}^{2p} dV_{K_i} \rightarrow \int_{S^{n-1}} \lin_{K,\xi}^{2p} dV_{K}$, and so the asserted identity is preserved in the limit.  
\end{proof}

\section{Finding a good direction} \label{sec:directions}

In order to use $\lin_{K,\xi}^2$ as an informative test-function, we will also need to find a good direction $\xi \in S^{n-1}$. This is handled in Theorem \ref{thm:directions} below, which extends Theorem \ref{thm:main-direction} from the Introduction. 

\begin{thm} \label{thm:directions}
Let $K \in \K$ be a convex body (having the origin in its interior). 
\begin{enumerate}
\item
There exists $\xi \in S^{n-1}$ so that:
\begin{equation} \label{eq:direction1}
\int_{S^{n-1}} \lin_{K,\xi}^4 dV_K \geq \frac{3n}{n+2} \frac{(\int_{S^{n-1}} \lin_{K,\xi}^2 dV_K)^2}{V(K)} . 
\end{equation}
\item
Assume moreover that $K$ is in $S_2$-isotropic position. \\
Then (\ref{eq:direction1}) holds in expectation over $\xi$ which is uniformly distributed in $S^{n-1}$:
\begin{equation} \label{eq:explicit-expectation}
\int_{\S^{n-1}} \int_{S^{n-1}} \lin_{K,\xi}^4 dV_K d\xi \geq \frac{3n}{n+2} \frac{\int_{S^{n-1}} (\int_{S^{n-1}} \lin_{K,\xi}^2 dV_K)^2 d\xi}{V(K)} . 
\end{equation}
Equality in (\ref{eq:explicit-expectation}) holds if and only if $K$ is a tangential body of a Euclidean ball. 
In particular, if $K \in \K^2_+$, or more generally, if the support of $S_K$ is the entire $S^{n-1}$, equality in (\ref{eq:explicit-expectation}) holds if and only if $K$ is a Euclidean ball. 
\item \label{it:assertion3}
For a general $K \in \K$, equality in  (\ref{eq:direction1}) holds for every $\xi \in S^{n-1}$, or equivalently,
\begin{equation} \label{eq:Gamma-multiple}
\Gamma_{-4} K = \brac{\frac{n+2}{3}}^{\frac{1}{4}} \Gamma_{-2} K ,
\end{equation}
 if and only if, up to a linear transformation, $K$ is a tangential body to a Euclidean ball and in addition $\Gamma_{-2} K$ and $\Gamma_{-4} K$ are Euclidean balls themselves. 
\end{enumerate}
\end{thm}

\begin{rem}
Recall from the Introduction that a tangential body to a Euclidean ball is equivalently characterized as having $h_K$ constant $S_K$-a.e. In that case, all the $L^p$ surface-area measures $S_p K$ are proportional, and so all $S_p$-isotropic positions coincide. The following characterization was obtained by Zou--Xiong in \cite{ZouXiong-IdenticalJohnAndLYZEllipsoids} for $p \in (0,\infty)$, and extended to $p=0$ (assuming the centroid of $K$ is at the origin) by Hu--Xiong \cite{HuXiong-LogJohnEllipsoid}, who showed that the following statements are equivalent:
\begin{itemize}
\item $K \in \K$ is a tangential body of a Euclidean ball in its $S_2$-isotropic (equivalently, $S_p$-isotropic) position.
\item $E_2 K$ (and in fact, any $E_p K$) coincides with the John ellipsoid $E_\infty K$. 
\item $E_2 K$ (and in fact, any $E_p K$) is a subset of $K$. 
\end{itemize}
\end{rem}

\begin{rem}\label{rem:L4-cosine}
The examples of a regular cube or octahedron show that there are bodies $K$ other than Euclidean balls which are tangential to a Euclidean ball and for which $S_2 K$ is isotropic; more generally, any tangential body to a ball which has enough symmetries will necessarily be in $S_2$-isotropic position. However, we do not know whether the property stated in assertion (\ref{it:assertion3}) above characterizes ellipsoids in the class of origin-symmetric convex bodies $\K_e$. This has to do with the fact that the $L^p$ spherical cosine-transform, while being injective on the space of even measures for $p > 0$ which is not an even integer \cite{Neyman-LpCosineInjectivity}, is not injective for $p \in 2 \N$ (since in that case $\{\abs{\scalar{\cdot,\xi}}^p\}$ span a finite dimensional linear space). Consequently, it is easy to construct examples of bodies $K \in \K_e$ different from Euclidean balls so that  $\Gamma_{-2} K$ and $\Gamma_{-4} K$ are themselves Euclidean balls (see e.g. \cite[Section 4]{RyaboginZvavitch-LpProjectionBodies}). But perhaps in conjunction with the assumption that $K$ is a tangential body of a ball, this already forces $K$ to be a ball itself? We leave this as an interesting open problem. 
\end{rem}

\begin{proof}[Proof of Theorem \ref{thm:directions}]
\hfill

\noindent \textbf{Reduction of (1) to (2).}
Observe that the first assertion on the existence of $\xi \neq 0$ so that (\ref{eq:direction1}) holds is centro-affine invariant. To see this, rewrite (\ref{eq:direction1}) as:
\[
\int_K \scalar{x^* , \xi}^4 dx \geq \frac{3n}{n+2} \frac{(\int_K \scalar{x^* , \xi}^2 dx)^2}{V(K)} . 
\]
Since $T(x)^* = T^{-*} x^*$ for any $T \in GL_n$, it follows that:
\[
\int_{T(K)} \scalar{y^* , \xi}^p dy = |\det(T)| \int_K \scalar{T(x)^* , \xi}^p dx =  |\det(T)| \int_K \scalar{x^* , T^{-1} \xi}^p dx ,
\]
and the centro-affine invariance immediately follows. Consequently, we may assume that $K$ is in $S_2$-isotropic position, thereby reducing the first assertion to the second. For convenience, let us also assume that $V(K)=1$. 

\medskip
\noindent \textbf{Proof of (2).}
Instead of randomly drawing $\xi \in S^{n-1}$ from the uniform Haar measure, it will be more convenient to draw $\xi \in \R^n$ from the (rotation invariant) standard Gaussian measure (since the assertion is homogeneous in $\xi$). 
 
If $\xi$ is a standard Gaussian random-vector in $\R^n$, it is straightforward to check that for all $u,v \in \R^n$:
\begin{equation} \label{eq:Gaussians}
\E \scalar{\xi,u}^4 = 3 |u|^4 ~,~ \E \scalar{\xi,u}^2 \scalar{\xi,v}^2 = 2 \scalar{u,v}^2 + |u|^2 |v|^2 . 
\end{equation}
Indeed, by rotation invariance of the Gaussian measure, denoting $X \sim Y$ if $X$ and $Y$ are identically distributed, we have:
\[
\scalar{\xi,u}^4 \sim |u|^4 \xi^4_1 ~ , ~  \scalar{\xi,u}^2 \scalar{\xi,v}^2 \sim \xi_1^2 \brac{\xi_1 \scalar{u,v} + \xi_2 \sscalar{u^{\perp},v}}^2 , 
\]
where $u^{\perp}$ is obtained by rotating $u$ by 90 degrees in the linear subspace spanned by $\{u,v\}$.  Noting that $\sscalar{u^{\perp},v}^2 = |u|^2 |v|^2 - \scalar{u,v}^2$, using that $\xi_1,\xi_2$ are independent standard Gaussian variables, and recalling that $\E \xi_i = 0$, $\E \xi_i^2 = 1$ and $\E \xi_i^4 = 3$, (\ref{eq:Gaussians}) immediately follows. 

Rewriting (\ref{eq:direction1}) as:
\[
\int_{S^{n-1}} \frac{\scalar{\theta^*,\xi}^4}{h_K^4(\theta^*)} dV_K(\theta^*) \geq \frac{3n}{n+2} \int_{S^{n-1} \times S^{n-1}}  \frac{\scalar{\theta^*_1,\xi}^2 \scalar{\theta^*_2,\xi}^2  }{h^2_K(\theta^*_1) h_K^2(\theta^*_2)}dV_K(\theta_1^*) dV_K(\theta_2^*)
\]
and taking expectation in $\xi$, our goal is thus to verify that:
\[
3 \int_{S^{n-1}} \frac{dV_K}{h_K^4} \geq \frac{3n}{n+2} \brac{ 2 \int_{S^{n-1} \times S^{n-1}} \scalar{\theta^*_1,\theta^*_2}^2 \frac{dV_K(\theta^*_1)}{h^2_K(\theta^*_1)}\frac{dV_K(\theta^*_2)}{h^2_K(\theta^*_2)} + \brac{\int_{S^{n-1}} \frac{dV_K}{h_K^2}}^2} . 
\]
Recalling that $\frac{dV_K}{h_K^2} = \frac{1}{n} \frac{dS_K}{h_K}= \frac{1}{n} dS_2 K$ and rearranging terms, this amounts to verifying:
\begin{align*}
&\Var_{V_K}(1 / h_K^2) \geq \\
&\frac{2}{n (n+2)} \brac{ \int_{S^{n-1} \times S^{n-1}} \scalar{\theta^*_1,\theta^*_2}^2 dS_2(K,\theta^*_1) dS_2(K,\theta^*_2) - \frac{1}{n} (\int_{S^{n-1}} dS_2 K)^2 } .
\end{align*}
And indeed, our assumption that $S_2 K$ is isotropic precisely ensures that the right-hand-side is $0$, and so we have verified (\ref{eq:direction1}) in expectation since:
\[
\Var_{V_K}(1 / h_K^2) \geq  0 . 
\]

The above analysis shows that equality in expectation in (\ref{eq:direction1}) is equivalent to:
\[
\Var_{V_K}(1 / h_K^2) = 0 ,
\]
meaning that $h_K$ is $V_K$ (and thus $S_K$) almost-everywhere constant, as asserted. 
By continuity of $h_K$, this means that $h_K$ is constant on the entire support of $S_K$, so whenever $S_K$ is of full-support, $h_K$ is constant on $S^{n-1}$ and so $K$ must be a Euclidean ball itself. 

\medskip
\noindent \textbf{Proof of (3).} Recalling (\ref{eq:Gamma-p-lin}), having equality in (\ref{eq:direction1}) for all $\xi \in S^{n-1}$ is clearly equivalent to (\ref{eq:Gamma-multiple}). As before, since (\ref{eq:Gamma-multiple}) is centro-affine invariant, we may apply a linear transformation so that $S_2 K$ is isotropic, or equivalently, so that $\Gamma_{-2} K$ is a Euclidean ball. 

If equality holds in (\ref{eq:direction1}) for all $\xi \in S^{n-1}$, then it also holds in expectation, and so by assertion (2), $K$ is a tangential body of a Euclidean ball. Equality for all $\xi$ also implies by (\ref{eq:Gamma-multiple}) that $\Gamma_{-4} K$ is a multiple of $\Gamma_{-2} K$, and thus a Euclidean ball as well. 

Conversely, if in its $S_2$-isotropic position $K$ is a tangential body of a Euclidean ball, then $h_K$ is $S_K$-a.e. constant and hence all the $S_p K$ measures are isotropic. In particular $V_K$ is isotropic, and by further scaling $K$, we may assume that it is a probability measure. We are also given that $\Gamma_{-4} K$ is a Euclidean ball, and hence $\Gamma_{-4} K = c \Gamma_{-2} K$ for some constant $c > 0$. Writing this as:
\[
n \int_{S^{n-1}} \frac{\scalar{\theta^*,\xi}^4}{h_K^4(\theta^*)} dV_K(\theta^*) = c^{-4}  n^2 \int_{S^{n-1} \times S^{n-1}}  \frac{\scalar{\theta^*_1,\xi}^2 \scalar{\theta^*_2,\xi}^2}{h^2_K(\theta^*_1) h_K^2(\theta^*_2)}  dV_K(\theta_1^*) dV_K(\theta_2^*) ,
\]
canceling the $V_K$-a.e. constant $h_K$,  integrating in $\xi$ as above, and using that $V_K$ is isotropic, we deduce :
\[
3 n = c^{-4} n^2 \brac{\frac{2}{n} + 1} .
\]
It follows that necessarily $c = \brac{\frac{n+2}{3}}^{\frac{1}{4}}$, and hence (\ref{eq:Gamma-multiple}) is established.
\end{proof}

\section{Proofs of Theorems \ref{thm:main1} and \ref{thm:main2}} \label{sec:main}

We are now ready to provide the proofs of Theorems \ref{thm:main1} and \ref{thm:main2}. While Theorem \ref{thm:main1} is a particular case of Theorem \ref{thm:main2}, it might be insightful to first give a proof of the former. 

\begin{proof}[Proof of Theorem \ref{thm:main1}]
Given $K \in \K^2_{+,e}$, there exists by Theorem \ref{thm:directions} a direction $\xi \in S^{n-1}$ so that (\ref{eq:direction1}) holds. 
Note that $\lin^2_{K,\xi} \in C^2_{e}$ is an even function (as $K$ is origin-symmetric). Using $z = \lin^2_{K,\xi}$ as a test function in the Rayleigh-Ritz characterization (\ref{eq:RR}) of $\lambda_{1,e}(-\Delta_K)$, we obtain, after employing Proposition \ref{prop:energy-lin} and (\ref{eq:direction1}):
\begin{align*}
\lambda_{1,e}(-\Delta_K) & \leq \frac{\int_{S^{n-1}} |\nabla (\lin^2_{K,\xi})|^2_{g_K} dV_K}{\int_{S^{n-1}} \lin^4_{K,\xi} dV_K - \frac{(\int_{S^{n-1}} \lin^2_{K,\xi} dV_K)^2}{V(K)}}  \\
& \leq \frac{\frac{4(n-1)}{3} \int_{S^{n-1}} \lin^4_{K,\xi} dV_K  }{\int_{S^{n-1}} \lin^4_{K,\xi} dV_K - \frac{n+2}{3 n} \int_{S^{n-1}} \lin^4_{K,\xi} dV_K  }  = 2 n . 
\end{align*}

If $\lambda_{1,e}(-\Delta_K) = 2n$ then the above argument shows that we cannot have strict inequality in (\ref{eq:direction1}) for any direction $\xi \in S^{n-1}$ (since this would imply that $\lambda_{1,e}(-\Delta_K) < 2n$). 
In other words:
\begin{equation} \label{eq:directions-reversed}
\int_{S^{n-1}} \lin_{K,\xi}^4 dV_K \leq \frac{3n}{n+2} \frac{(\int_{S^{n-1}} \lin_{K,\xi}^2 dV_K)^2}{V(K)} \;\;\; \forall \xi \in S^{n-1}. 
\end{equation}
By centro-affine invariance we may assume that $K$ is in $S_2$-isotropic position. By Theorem \ref{thm:directions} (2), we know that the reverse inequality in (\ref{eq:directions-reversed}) holds in expectation over $\xi$ which is uniformly distributed in $S^{n-1}$. As the expressions in (\ref{eq:directions-reversed}) are continuous in $\xi$, it follows that we must have equality in (\ref{eq:directions-reversed}) for all $\xi$. Since $K \in \K^{2}_+$, assertions (2) or (3) of Theorem \ref{thm:directions} then imply that $K$ is a Euclidean ball in its $S_2$-isotropic position, thereby verifying that $K$ is an ellipsoid. 
\end{proof}

\begin{proof}[Proof of Theorem \ref{thm:main2}]
Let $K \in \K_e$ and let $\frac{1}{R} B_2^n \subset K$ for some $R > 0$. 
By Theorem \ref{thm:directions}, there exists a direction $\xi \in S^{n-1}$ so that (\ref{eq:direction1}) holds. 
Again,  $\lin^2_{K,\xi} \in C_{e}$ is an even function. By Lemma \ref{lem:diff}, $h_{K_{R,\xi}} := h_K(R^2 + \lin^2_{K,\xi})$ is the support function of an origin-symmetric convex body $K_{R,\xi} \in \K_e$. Using $L = K_{R,\xi}$ as a test convex body in the definition (\ref{eq:def-lambdaC}) of $\lambdaC_{1,e}(K)$, we deduce that:
\[
\lambdaC_{1,e}(K) \leq (n-1) \frac{ \int_{S^{n-1}} (R^2 + \lin^2_{K,\xi})^2 dV_K - V_K(R^2+\lin^2_{K,\xi} ; 2) }{ \int_{S^{n-1}} (R^2 + \lin^2_{K,\xi})^2 dV_K - \frac{V_K(R^2+\lin^2_{K,\xi}; 1)^2}{V(K)} } . 
\]
It is immediate to check that both numerator and denominator above do not depend on the value of $R$ -- simply use the multi-linearity of mixed volumes:
\begin{align*}
V_K(R^2 +\lin^2_{K,\xi} ; 2) & = R^4 V(K) + 2 R^2 V_K(\lin^2_{K,\xi}; 1) + V_K(\lin^2_{K,\xi} ; 2) ,\\
V_K(R^2 +\lin^2_{K,\xi} ; 1) & = R^2 V(K) + V_K(\lin^2_{K,\xi} ; 1) ,
\end{align*}
plug this into the above expression, and recall that $V_K(\lin^2_{K,\xi} ; 1) = \int_{S^{n-1}} \lin^2_{K,\xi} dV_K$. Consequently:
\[
\lambdaC_{1,e}(K) \leq (n-1) \frac{ \int_{S^{n-1}} \lin^4_{K,\xi} dV_K - V_K(\lin^2_{K,\xi} ; 2) }{ \int_{S^{n-1}} \lin^4_{K,\xi} dV_K - \frac{(\int_{S^{n-1}} \lin_{K,\xi}^2 dV_K)^2}{V(K)} } ,
\]
which by Proposition \ref{prop:V2-lin} and (\ref{eq:direction1}) translates to:
\begin{align*}
& =  (n-1) \frac{ \frac{4}{3} \int_{S^{n-1}} \lin^4_{K,\xi} dV_K}{ \int_{S^{n-1}} \lin^4_{K,\xi} dV_K - \frac{(\int_{S^{n-1}} \lin_{K,\xi}^2 dV_K)^2}{V(K)} } \\
& \leq \frac{\frac{4(n-1)}{3} \int_{S^{n-1}} \lin^4_{K,\xi} dV_K  }{\int_{S^{n-1}} \lin^4_{K,\xi} dV_K - \frac{n+2}{3 n} \int_{S^{n-1}} \lin^4_{K,\xi} dV_K  }  = 2 n .
\end{align*}
Note that $\int_{S^{n-1}} \lin^4_{K,\xi} dV_K > 0$ since $S_K$ is not concentrated in any hemisphere. 

Since an ellipsoid $\EE$ is in $\K^{2}_{+,e}$, we know that $\lambdaC_{1,e}(\EE) = \lambda_{1,e}(-\Delta_{\EE}) = \lambda_{1,e}(-\Delta_{B_2^n}) = 2n$. Conversely, assume that $\lambdaC_{1,e}(K) = 2n$. As in the proof of Theorem \ref{thm:main1}, the above argument shows that we cannot have strict inequality in (\ref{eq:direction1}) for any direction $\xi \in S^{n-1}$, and so arguing as before, it follows that we must have equality in (\ref{eq:direction1}) for all $\xi$. Theorem \ref{thm:directions} (3) then implies that $K$ is the linear image of a tangential body of a Euclidean ball. 
\end{proof}

\begin{rem} \label{rem:tangential}
Theorem \ref{thm:directions} (3) gives the added information when $\lambdaC_{1,e}(K) = 2n$ that, up to a linear transformation, $K$ is a tangential body of a Euclidean ball with $\Gamma_{-2} K$ and $\Gamma_{-4} K$ being Euclidean balls themselves. As mentioned in Remark \ref{rem:L4-cosine}, we do not know if this is enough to conclude that $K$ is an ellipsoid. However, we can do a bit more and show that in fact (up to a linear transformation) $K$ is an $(n-2)$-tangential body of a Euclidean ball. Since this is not a very significant addition, we only sketch the argument and leave the details to the reader. 

A convex body $K$ is called an $m$-tangential body of a second convex body $L \subset K$ if each $n-m-1$-extreme support plane of $K$ is also a support plane of $L$. Given $u \in S^{n-1}$, the support plane $\{ x \; ; \; \scalar{x,u} = h_K(u) \}$ of $K$ and the normal vector $u$ are called $r$-extreme if there do not exist $r+2$ linearly independent normal vectors $u_1,\ldots,u_{r+2}$ at one and the same boundary point of $K$ such that $u = u_1 + \ldots + u_{r+2}$ \cite[pp. 85-86]{Schneider-Book-2ndEd}. We denote by $\extr_r K$ the collection of $r$-extreme normal vectors $u \in S^{n-1}$ of $K$. Clearly $\extr_0 K \subset \ldots \subset \extr_{n-1} K = S^{n-1}$. By \cite[Theorem 4.5.3]{Schneider-Book-2ndEd}, $\supp(S_K)$, the support of $S_K$, is precisely the closure of $\extr_0 K$, and so a $(n-1)$-tangential body is simply called a tangential body. 
We claim that if  $\lambdaC_{1,e}(K) = 2n$ and $K \in \K_e$ is in $S_2$-isotropic position, then necessarily $\extr_1 K \subset \supp(S_K)$, and hence by (\ref{eq:intro-tangential-body}) $K$ is an $(n-2)$-tangential body of a Euclidean ball. Assume otherwise, and consider an even non-negative smooth function $g$ compactly supported in $S^{n-1} \setminus \supp(S_K)$ so that $g(u) = 1$ for some $u \in \extr_1 K$. The idea is now to repeat the argument above using $L_{\eps} \in \K_e$ with $h_{L_{\eps}} = h_K ((R+1)^2 + \lin_{K,\xi}^2 + \eps g)$ for small enough $\eps > 0$ as a test body in the definition (\ref{eq:def-lambdaC}) of $\lambdaC_{1,e}(K)$. 
Since $g$ is supported outside of $\supp(S_K)$, it will not have any effect on the $V(L_\eps[1],K[n-1])$ or $\int (\frac{h_{L_\eps}}{h_K})^2 dV_K$ terms appearing in (\ref{eq:def-lambdaC}), and will only affect the second mixed volume $V(L_\eps[2],K[n-2])$. By monotonicity of mixed volumes we have $V(L_\eps[2],K[n-2]) \geq V(L_0[2], K[n-2])$ since $\eps > 0$. It remains to show that we have strict inequality $V(L_\eps[2],K[n-2]) > V(L_0[2], K[n-2])$, as this will imply $\lambdaC_{1,e} < 2n$ and yield the desired contradiction. The latter follows from the recent resolution of the equality cases in Minkowski's quadratic inequality by Shenfeld--van Handel \cite[Theorem 2.2]{ShenfeldVanHandel-MinkowskiQuadratic}, since for equality to occur, the supporting planes to $L_0$ and $L_\eps$ must coincide for every $1$-extreme supporting plane of $K$; as $g(u) = 1$ for $u \in \extr_1 K$ this is not the case, and it follows that we must have strict inequality. 
\end{rem}

\begin{rem}
As suggested to us by a referee, it is worth noting that our proof almost does not make use of any particular feature of the Hilbert--Brunn--Minkowski operator $\Delta_K$. Proposition \ref{prop:V2-lin} remains valid for an eigenfunction $\Psi$ of $-\Delta_{g,\mu}$ with eigenvalue $n-1$ for any weighted Laplacian $\Delta_{g,\mu}$ on a weighted Riemannian manifold $(M,g,\mu)$. Whenever the corresponding eigenspace $\{\Psi_\xi\}$ can be shown to satisfy some type of isotropicity condition in $\xi$ which permits to establish an analogue of (\ref{eq:explicit-expectation}), then there exists some $\xi$ so that $\Psi_\xi^2$ has Rayleigh quotient at most $2 n$. In our setting, the isotropicity could be ensured by the centro-affine invariance of $\Delta_K$ and the linearity of the eigenfunction $\Psi_\xi = \lin_{K,\xi}$ in $\xi$. 
\end{rem}

\section{Implications for the even $L^p$-Minkowski problem} \label{sec:LpMinkowski}

In this final section, we fill some missing details from the discussion in Subsection \ref{subsec:intro-LpMinkowski} from the Introduction. 

Recalling the definition  (\ref{eq:intro-F}) of $F_{\mu,p}$, let us introduce when $p \neq 0$:
\[
G_{\mu,p}(K) := \log (p F_{\mu,p}(K)) = \log (\int h_K^p d\mu) - \frac{p}{n} \log V(K) . 
\]
Note that $\frac{1}{p} G_{\mu,p}(K)$, $F_{\mu,p}(K)$ and $F_{c \cdot \mu,p}(K)$ have identical critical points and (local) minima / maxima for any $c > 0$. For completeness, we mention that when $p=0$ these definitions should be interpreted in the limiting sense as in \cite{BLYZ-logMinkowskiProblem}, namely:
\[
F_{\mu,0}(K) :=  \frac{\exp(\int \log h_K d\tilde \mu)}{V(K)^{\frac{1}{n}} } ~,~ G_{\mu,0}(K) := \int \log h_K d\tilde\mu - \frac{1}{n} \log V(K) ,
\]
where $\tilde \nu$ denotes the normalized measure $\nu/ \norm{\nu}$.

\medskip

Given $K \in \K^{2}_{+,e}$ and $z \in C^2_e(S^{n-1})$, we define the $m$-th $C^2_e$-variation ($m=1,2$) at $K$ in the direction of $z$ of a nice-enough functional $F: \K^{2}_{+,e} \rightarrow \R$ as:
\[
\delta^m_K F(z) := \left . \brac{\frac{d}{d\eps}}^m  \right |_{\eps= 0} F(K_\eps) ~,~ h_{K_\eps} = h_K(1+\eps z) . 
\]
Note that since $K \in \K^2_{+,e}$ then for $|\eps|$ small enough, $D^2 h_{K_{\eps}} > 0$, and hence $h_{K_\eps}$ is indeed the support function of a convex body $K_\eps \in \K^{2}_{+,e}$.

\begin{proof}[Proof of Proposition \ref{prop:intro-local-minimum}]
Recall that $K \in \K^{2}_{+,e}$ and $z \in C^2_e(S^{n-1})$.
Thanks to multi-linearity of mixed volumes it is immediate to verify:
\[
\delta^1_K V(z) = n V_K(z) = n \int z dV_K  ~,~ \delta^2_K V(z) = n (n-1) V_K(z,z) . 
\]
\begin{enumerate}
\item
Let $\mu$ be a non-zero finite even Borel measure on $S^{n-1}$ and let $p \neq 0$. The first $C^2_e$-variation of $G_{\mu,p}$ is then:
\[
\delta^1_K G_{\mu,p}(z) = \frac{ p \int z h_K^p d\mu}{\int h_K^p d\mu} - \frac{p}{n} \frac{ n \int z dV_K}{V(K)} . 
\]
Consequently, we see that $\delta^1_K F(\mu,p) \equiv 0$ on $C^2_e(S^{n-1})$ iff $\delta^1_K G(\mu,p) \equiv 0$ on $C^2_e(S^{n-1})$ iff:
\[
0 = \int z \brac{\widetilde{(h_K^p d\mu)} - \widetilde{dV_K}} \;\;\; \forall z \in C^2_e(S^{n-1}) .
\]
Hence, by density of $C^2_e(S^{n-1})$ in $C_e(S^{n-1})$, since $\mu$ is even and $K$ is origin-symmetric we see that $K$ is a $C^2_e$-critical point of $F_{\mu,p}$ iff $V_K = c h_K^p \mu$ for some $c > 0$, or equivalently iff $S_p K = c n \mu$. 
\item
The second $C^2_e$-variation is:
\[
\delta^2_K G_{\mu,p}(z) = \frac{p(p-1) \int z^2 h_K^p d\mu}{\int h_K^p d\mu} - \frac{p^2 \brac{\int z h_K^p d\mu}^2}{\brac{\int h_K^p d\mu}^2} - \frac{p}{n} \brac{ \frac{n(n-1) V_K(z,z)}{V(K)} - \brac{ \frac{n \int z dV_K}{V(K)} }^2   } . 
\]
When $\mu = c \cdot S_p K$ (so that $h_K^p \mu = c n V_K$), this simplifies to:
\begin{align*}
\frac{V(K)}{p} \delta^2_K & G_{c \cdot S_p K,p}(z) = (p-1) \int z^2 dV_K - p \frac{(\int z dV_K)^2}{V(K)} - (n-1) V_K(z,z) + n \frac{(\int z dV_K)^2}{V(K)} \\
& = (n-1) \brac{ \int z^2 dV_K - V_K(z,z)} - (n-p) \brac{ \int z^2 dV_K - \frac{(\int z dV_K)^2}{V(K)}} \\
& = \int (-\Delta_K z) z dV_K - (n-p) \brac{ \int z^2 dV_K - \frac{(\int z dV_K)^2}{V(K)}} ,
\end{align*}
where we used (\ref{eq:Delta-V2}) in the last transition. It follows that, regardless of the sign of $p$, $\delta^2_K F_{S_p K ,p} \geq 0$ on $C^2_e(S^{n-1})$ iff $\frac{1}{p} \delta^2_K G_{S_p K ,p} \geq 0$ on $C^2_e(S^{n-1})$ iff $\lambda_{1,e}(-\Delta_K) \geq n-p$. A completely analogous proof holds when $p=0$, which we leave for the reader to verify.
\item
Finally, since $-\Delta_K$ is an unbounded operator, we see that we could never have $\frac{1}{p} \delta^2_K G_{S_p K ,p} \leq 0$, or equivalently $\delta^2_K F_{S_p K ,p} \leq 0$. Consequently, if $F_{\mu,p}$ had a local maximum at $K$, then by part (1), since $K$ is a critical point of $F_{\mu,p}$, then necessarily $S_p K = c \cdot \mu$ for some $c > 0$. Since we cannot have $\delta^2_K F_{S_p K ,p} \leq 0$, it follows that $K$ cannot be a local maximum point afterall. 
\end{enumerate}
 \end{proof}

\begin{proof}[Proof of Corollary \ref{cor:intro-local-minimum}]
Since $K_2 \in \K_e$ is a local minimum point of $F_{S_p K_1,p}$, Proposition \ref{prop:intro-EL} implies $S_p K_2 = c \cdot S_p K_1$ for some $c > 0$, and by rescaling $K_2$ we may ensure that $c = 1$ without altering its local minimality (by $0$-homogeneity of $F_{\mu,p}$). If $K_2 \notin K^2_{+,e}$,
 then already $K_2 \neq K_1$. Otherwise $K_2 \in \K^2_{+,e}$, and so it is a local minimum point of $F_{S_p K_2,p}$ under $C^2_e$ variations. It follows by Proposition \ref{prop:intro-local-minimum} (2) that $\lambda_{1,e}(-\Delta_{K_2}) \geq n-p$, and hence $K_2$ must differ from $K_1$ for which $\lambda_{1,e}(-\Delta_{K_1}) < n-p$. 
\end{proof}

\begin{proof}[Proof of Theorem \ref{thm:non-unique1}]
This follows immediately from Corollary \ref{cor:intro-local-minimum} and Theorem \ref{thm:main1}, by setting $q(K_1) := n - \lambda_{1,e}(-\Delta_{K_1}) \in (-n,1)$. 
\end{proof}

\begin{proof}[Proof of Theorem \ref{thm:non-unique2}]
This follows since $\underline{\lambda}_{1,e}(Q^n) = n$ by \cite[Theorem 10.2]{KolesnikovEMilman-LocalLpBM}, and so we may choose $Q_i \in \K^2_{+,e}$ so that $n - \lambda_{1,e}(-\Delta_{Q_i})$ is arbitrarily close to $0$, and conclude by Corollary \ref{cor:intro-local-minimum}.
\end{proof}

Before proceeding with the proof of Proposition \ref{prop:intro-diam}, and as already alluded to in the proof of Corollary \ref{cor:intro-local-minimum}, we first recall that, regardless of the value of $p$, any solution $K \in \K$ to:
\begin{equation} \label{eq:regularity}
S_p K = f \Leb ~,~  f \in C^\alpha(S^{n-1}) ~,~ f > 0 ,
\end{equation}
 necessarily satisfies $K \in \K^2_{+}$. Indeed, since $h_K$ is a-priori assumed to be positive (as the origin lies in the interior of $K \in \K$),  
the standard regularity theory implies that $h_{K} \in C^{2,\alpha}(S^{n-1})$ (see Caffarelli \cite{Caffarelli-StrictConvexity,CaffarelliHigherHolderRegularity}, as well as \cite[Proposition 1.2]{ChouWang-LpMinkowski} or \cite[Theorem 1.1]{BBC-SmoothnessOfLpMinkowski}). 
  In addition, $K$ must also have strictly positive curvature since $\det(D^2 h_{K}) = f h_K^{p-1} > 0$, and the assertion follows.

\begin{proof}[Proof of Proposition \ref{prop:intro-diam}]
Since $K_p \in \K_e$ is a local minimum point of $F_{\mu,p}$, we know by Proposition \ref{prop:intro-EL} that $S_p K_p = c_p \cdot \mu = c_p f \Leb$ for some $c_p > 0$. Our assumptions on $f$ and the regularity theory for (\ref{eq:regularity}) discussed above imply that $K_p \in \K^{2}_{+,e}$. Since  $\delta^2_{K_p} F_{S_p K_p,p} \geq 0$, Proposition \ref{prop:intro-local-minimum} (2) verifies $\lambdaC_{1,e}(K_p) = \lambda_{1,e}(-\Delta_{K_p}) \geq n-p$. Assume in the contrapositive that there is a sequence of $p_i \searrow -n$ with $d_G(K_{p_i},B_2^n) \leq C < \infty$. By the Blaschke selection theorem there exists a subsequence, which we continue to denote by $\{p_i\}$, and positive scaling coefficients $\{R_i\}$, so that $\tilde K_{p_i} := K_{p_i}/R_i$ are sandwiched between $B_2^n$ and $C B_2^n$, and converge in the Hausdorff metric to a compact set $\tilde K_{-n}$, which is clearly in $\K_e$.  Consequently $S_{\tilde K_{p_i}}$ weakly converges to $S_{\tilde K_{-n}}$, and hence
$S_{p_i} \tilde K_{p_i} = R_i^{p_i - n} S_{p_i} K_{p_i} = R_i^{p_i - n} c_{p_i} \mu$ weakly converges to $S_{-n} \tilde K_{-n}$. It follows that $c_{-n} := \lim_{i \rightarrow \infty} R_i^{p_i -n} c_{p_i}$ exists and is in $(0,\infty)$, and that $S_{-n} \tilde K_{-n} = c_{-n} \mu$. In particular, we deduce as before that $\tilde K_{-n} \in \K^2_{+,e}$, and the assumption that $f$ is non-constant implies that $\tilde K_{-n}$ is not an ellipsoid by (\ref{eq:ellipsoids-curvature}). 

On the other hand, by upper semi-continuity of $\lambdaC_{1,e}$ with respect to $C$-convergence, it follows that 
\[
\lambda_{1,e}(-\Delta_{\tilde K_{-n}}) = \lambdaC_{1,e}(\tilde K_{-n}) \geq \limsup_{i \rightarrow \infty} \lambdaC_{1,e}(\tilde K_{p_i}) = \limsup_{i \rightarrow \infty} \lambdaC_{1,e}( K_{p_i}) \geq 2n .
\]
As $\tilde K_{-n}$ is not an ellipsoid, this contradicts the equality case of Theorem \ref{thm:main1}, concluding the proof. 
\end{proof}

\begin{proof}[Proof of Theorem \ref{thm:intro-super-critical}]
\hfill
\begin{enumerate}
\item
Proposition \ref{prop:intro-EL} (3) already shows that $F_{\mu,p}$ cannot have any local maximum points which are in $\K^2_{+,e}$. 
\item
Assume in the contrapositive that $K \in \K^2_{+,e}$ is a local minimum of $F_{\mu,p}$. By Proposition \ref{prop:intro-EL} (or by Proposition \ref{prop:intro-local-minimum} (1)) we deduce that $S_p K = c \cdot \mu$ for some $c > 0$, and since $\delta^2_K F_{S_p K, p} \geq 0$, Proposition \ref{prop:intro-local-minimum} (2) implies that $\lambda_{1,e}(-\Delta_{K}) \geq n-p \geq 2n$. But this is impossible by Theorem \ref{thm:main1} unless $p=-n$ and $K$ is an ellipsoid $\EE$, in which case $c \cdot \mu = S_{-n} \EE = c' \Leb$ by (\ref{eq:ellipsoids-curvature}). 
\item
Assume in the contrapositive that $K \in \K_e$ is a local minimum of $F_{\mu,p}$. By Proposition \ref{prop:intro-EL} $S_p K = c \cdot \mu$ for some $c > 0$, and by the standard regularity theory for (\ref{eq:regularity}), it follows that $K \in \K^2_{+,e}$. The assertion now follows by part (2). 
\item
Write  $p = -(n + a)$ and $q = -(n + b)$ with $a > b > 0$. 
When $\mu = f \Leb$ with $\norm{f}_{L^{-\frac{n}{b}}(\Leb)} \in (0,\infty)$ then by the reverse H\"older inequality and polar integration on $K^{\circ}$:
\begin{align*}
 -(n+a) V(K)^{-\frac{n+a}{n}} F_{\mu,p}(K) & = \int_{S^{n-1}}  \frac{f}{h_K^{n+a}} d\Leb \geq \brac{\int_{S^{n-1}} f^{-\frac{n}{b}} d\Leb}^{-\frac{b}{n}} \brac{\int_{S^{n-1}} \frac{d\Leb}{h_K^{\frac{n+a}{n+b}n}}}^{\frac{n+b}{n}} \\
&  = \norm{f}_{L^{-\frac{n}{b}}(\Leb)} \brac{ c_{n,a,b}  \int_{K^{\circ}} \abs{x}^{\frac{a-b}{n+b} n} dx}^{\frac{n+b}{n}} . 
\end{align*}

Since $F_{\mu,p}$ is $0$-homogeneous with respect to scaling, to show coercivity of $-F_{\mu,p}$ under a volume constraint we may assume that $V(K)=V(B_2^n)$. By the reverse Blaschke--Santal\'o inequality due to Bourgain and V.~Milman \cite{BourgainMilman-ReverseSantalo}, $V(K^{\circ})$ is bounded below. Consequently, if in addition $d_G(K,B_2^n) = d_G(K^{\circ},B_2^n) \rightarrow \infty$ then it is elementary to see (using e.g. \cite[Lemma 2.2]{Klartag-LowM}) that the right-hand-side above tends to infinity as well, uniformly in $d_G(K,B_2^n)$. 

Since $\K_e \ni K \mapsto -F_{\mu,p}(K)$ is continuous with respect to the Hausdorff topology, and since $\{ K \in \K_e \; ; \; V(K) = V(B_2^n) ~,~ d_G(K,B_2^n) \leq C \}$ is compact, it follows that $F_{\mu,p}$ necessarily attains a global maximum $K_{\max} \in \K_e$. By part (1), we see that this maximum cannot be attained on $\K^2_{+,e}$. Finally, if $\mu$ satisfies the condition in (3), we cannot have that $S_p K_{\max} = c \cdot \mu$, since otherwise the regularity theory for (\ref{eq:regularity}) would imply that $K_{\max} \in \K^2_{+,e}$. 
\end{enumerate}
\end{proof}

\begin{rem} \label{rem:no-EL}
Applying Theorem \ref{thm:intro-super-critical} (4) to $\mu = \Leb$, one can deduce the well-known fact that there can be no $\K^2_{+,e}$ minimizer to the Mahler volume product $V(K) V(K^{\circ})$ (see \cite{RSW-NoCurvatureInMahler} for a much stronger result).  
\end{rem}

\begin{rem}
Our condition in Theorem \ref{thm:intro-super-critical} (4) which ensures that there is no global minimum to $F_{\mu,p}$, namely $\norm{f}_{L^{\frac{n}{n+q}}(\Leb)} \in (0,\infty)$ for some $q \in (p,-n)$, complements the conditions in \cite[Theorem 1.5]{BBCY-SubcriticalLpMinkowski} and \cite[Theorem 5.1]{LuWang-CriticalAndSupercriticalLpMinkowski} which do ensure the existence of a global minimum in the subcritical and supercritical regimes, respectively. 
\end{rem}

\begin{rem} \label{rem:symmetries}
It is interesting to compare our Proposition \ref{prop:intro-diam} and Theorem \ref{thm:intro-super-critical} (3) with \cite[Lemma 4.6]{LuWang-CriticalAndSupercriticalLpMinkowski} and \cite[Section 3]{JianLuZhu-UnconditionalCriticalLpMinkowski}, respectively, since at first sight they seem contradictory to each other. In those papers, the authors were able to obtain existence results in the centro-affine Minkowski problem for the critical exponent $p=-n$ under certain symmetry assumptions on the density $f$ of $\mu$ and the solution $K \in \K^2_{+,e}$. Denoting by $\K_G \subset \K_e$ the subset of convex bodies invariant under a fixed group of symmetries $G$ (which includes origin-symmetry), these authors considered a minimizing sequence $K_i \in \K_G$ for the $L^{-n+\delta_i}$ and $L^{-n}$-Minkowski problems, respectively, and showed that under appropriate conditions, $K_i$ converge in the Hausdorff metric to the desired solution $K \in \K_G$. The point is that these $K_i$ are not global nor local minimizers on $\K_e$ but only on $\K_G$, and hence our results for local minimizers on $\K_e$ do not apply. 
\end{rem}

\bibliographystyle{plain}
\bibliography{../../../ConvexBib}

\end{document}